\documentclass{amsart}

\usepackage{hyperref}
\usepackage{amsmath}%,natbib}
\usepackage[numbers, square]{natbib}

\usepackage[pdftex]{graphicx}

\usepackage{amssymb, bbm}
\usepackage{ifthen}
\usepackage{enumitem}
\usepackage{grffile}
\usepackage{amsthm}

\theoremstyle{plain}

\newtheorem{satz}{Satz}[section]

\newtheorem{prop}[satz]{Proposition}
\newtheorem{theo}[satz]{Theorem}
\newtheorem{cor}[satz]{Corollary}

\theoremstyle{definition}
\newtheorem{ex}[satz]{Example}

\newtheorem{definition}[satz]{Definition}

\theoremstyle{remark}
\newtheorem{rem}[satz]{Remark}

\newcommand{\ind}{\mathbbm{1}}

\newcommand{\R}{\mathbb R}
\newcommand{\E}{\mathbb E}

\newcommand{\N}{\mathbb N}

\renewcommand{\P}{\mathbb P}

\DeclareMathOperator{\Var}{Var}

\DeclareMathOperator{\sign}{sign}
\graphicspath{{./plots/}}

\newcommand{\todistr}{\overset{d}{\underset{n\to\infty}\longrightarrow}}
\newcommand{\toweak}{\overset{w}{\underset{n\to\infty}\longrightarrow}}

\newcommand{\ton}{\overset{}{\underset{n\to\infty}\longrightarrow}}
\newcommand{\eps}{\varepsilon}

\newcommand{\eqfdd}{\overset{f.d.d.}{\underset{}=}}
\newcommand{\eqdistr}{\overset{f.d.d.}{\underset{}=}}
\newcommand{\eqd}{\overset{d}{\underset{}=}}

\newcommand{\dd}{{\rm d}}
\newcommand{\eee}{{\rm e}}

\begin{document}
\title[Max-stable processes and stationary systems of L\'evy particles]{Max-stable processes and stationary systems of L\'evy particles}
\author{Sebastian Engelke}
\address{Sebastian Engelke, Ecole Polytechnique F\'ed\'erale de Lausanne,
EPFL-FSB-MATHAA-STAT, Station 8, 1015 Lausanne, Switzerland\newline
Universit\'{e} de Lausanne, Quartier UNIL-Dorigny,
B\^{a}timent Extranef, 1015 Lausanne, Switzerland
}
\email{sebastian.engelke@epfl.ch}
\author{Zakhar Kabluchko}
\address{Zakhar Kabluchko, Institute of Mathematical Statistics, University of M\"unster, Orl\'{e}ans-Ring 10, 48149 M\"unster, Germany}
\email{zakhar.kabluchko@uni-muenster.de}


\begin{abstract}
We study stationary max-stable processes $\{\eta(t)\colon t\in\mathbb R\}$ admitting a representation of the form $\eta(t)=\max_{i\in\mathbb N}(U_i+ Y_i(t))$, where $\sum_{i=1}^{\infty} \delta_{U_i}$ is a Poisson point process on $\mathbb R$ with intensity ${\rm e}^{-u} {\rm d} u$, and $Y_1,Y_2,\ldots$ are i.i.d.\ copies of a process $\{Y(t)\colon t\in\mathbb R\}$ obtained by running a L\'evy process for positive $t$ and a dual L\'evy process for negative $t$. We give a general construction of such L\'evy--Brown--Resnick processes, where the restrictions of $Y$ to the positive and negative half-axes are L\'evy processes with random birth and killing times. We show that these max-stable processes appear as limits of suitably normalized pointwise maxima of the form $M_n(t)=\max_{i=1,\ldots,n} \xi_i(s_n+t)$, where $\xi_1,\xi_2,\ldots$ are i.i.d.\ L\'evy processes and $s_n$ is a sequence such that $s_n\sim c \log n$ with $c>0$. Also, we consider maxima of the form $\max_{i=1,\ldots,n} Z_i(t/\log n)$, where $Z_1,Z_2,\ldots$ are i.i.d.\ Ornstein--Uhlenbeck processes driven by an $\alpha$-stable noise with skewness parameter $\beta=-1$. After a linear normalization, we again obtain limiting max-stable processes of the above form. This gives a generalization of the results of Brown and Resnick [Extreme values of independent stochastic processes, \textit{J.\ Appl.\ Probab.}, 14 (1977), pp.\ 732--739] to the totally skewed $\alpha$-stable case.
\end{abstract}
\maketitle

\section{Statement of results}
\subsection{Introduction}\label{subsec:intro}
Max-stable stochastic processes form a widely used class of models for extremal phenomena in space and time. The one-dimensional margins of max-stable processes belong to the family of extreme-value distributions.  For the purposes of the present paper, it will be convenient to choose the marginal distribution functions to be of the standard Gumbel form $\exp(-\eee^{-x})$, $x\in\R$. Our processes will be defined on $T=\R$. With these conventions, a stochastic process $\{\eta(t)\colon t\in \R\}$ is called \textit{max-stable} if for every $n\in\N$,
\begin{equation}\label{eq:max_stable_def}
\left\{\max_{i=1,\ldots,n} \eta_i(t) - \log n \colon t\in \R\right\} \eqfdd  \{\eta(t)\colon t\in \R\},
\end{equation}
where $\eta_1,\ldots,\eta_n$ are i.i.d.\ copies of the process $\eta$.
By a result of~\citet{deh1984}, any max-stable process $\eta$ admits a \textit{spectral representation} of the form
\begin{equation}\label{de_haan_rep}
\{\eta(t)\colon t\in \R\} \eqfdd \left\{\max_{i\in\N}\, (U_i + Y_i(t))\colon t\in\R\right\},
\end{equation}
where
\begin{itemize}
\item %[(a)] 
$\sum_{i=1}^{\infty} \delta_{U_i}$ is a Poisson point process (PPP) on $\R$ with intensity $\eee^{-u}\dd u$;
\item %[(b)]
$Y_1,Y_2,\ldots$ are i.i.d.\ copies of a stochastic process $Y=\{Y(t)\colon t\in\R\}$ which takes values in $\R \cup\{-\infty\}$ and satisfies the condition $\E \eee^{Y(t)} = 1$;
\item %[(c)] 
$\sum_{i=1}^{\infty} \delta_{U_i}$ is independent of $\{Y_i\colon i\in\N\}$.
\end{itemize}
As usual, $\delta_u$ denotes the unit Dirac measure at $u$.

In the special case when $Y(0)=0$ a.s.\ it is convenient to imagine an infinite system of particles on $\R\cup\{-\infty\}$ that start at time $t=0$ at the spatial positions $U_i$ and move independently according to the law of the process $Y$. Then, $\eta(t)$ is just the position of the right-most particle at time $t$. In the case when $Y(0)$ is not $0$, the starting positions of the particles are at $U_i+Y_i(0)$. If, for some $t\in\R$, $Y_i(t)$ becomes $-\infty$, the particle $i$ is considered as ``killed'' at time $t$.

In this paper, we will be interested in \textit{stationary} max-stable processes. One of the interesting features of the de Haan representation~\eqref{de_haan_rep} is that the process $\eta$ can be stationary even though the process $Y$ is not. The first example of this type was constructed by~\citet{bro1977}.
They considered a process of the form
\begin{align}\label{BR}
  \eta_{\text{BR}}(t) = \max_{i\in\N}\, (U_i + B_i(t) - |t|/2),
\end{align}
where $B_1,B_2,\ldots$ are independent copies of a two-sided standard Brownian motion $\{B(t)\colon  t\in\R\}$. \citet{bro1977} observed that the process $\eta_{\text{BR}}$ is stationary and max-stable. Also, they showed that $\eta_{\text{BR}}$ appears as the large $n$ limit for pointwise maxima of
\begin{enumerate}
 \item[(a)] $n$ independent Brownian motions and
 \item[(b)] $n$ independent Ornstein--Uhlenbeck processes,
\end{enumerate}
after appropriate normalization which involves spatial rescaling of the processes.  Note that statement (b) explains the stationarity of $\eta_{\text{BR}}$.

Since the Brownian motion $B$ is both a Gaussian process and a L\'evy process, it is natural to ask whether there is a generalization of the Brown--Resnick process $\eta_{\text{BR}}$ in which the spectral functions $Y_i$ are i.i.d.\
\begin{enumerate}
\item [(i)] Gaussian processes or
\item [(ii)] L\'evy processes.
\end{enumerate}

Regarding question~(i), it was shown in~\cite{kab2009} that if $W_1,W_2,\ldots$ are i.i.d.\ copies of a centered Gaussian process $W$ with stationary increments and variance $\sigma^2(t)=\Var W(t)$, then the max-stable process
$$
\eta(t) := \max_{i\in\N}\, (U_i + W_i(t) - \sigma^2(t)/2)
$$
is stationary. This class of max-stable processes has become a common tool in spatial extreme value modeling~\cite{dav2012b, eng2014}.

In this paper, we will be interested in question~(ii). Max-stable processes whose spectral functions are L\'evy processes were first considered by~\citet{sto2008}. Our aim is to describe a two-sided version of Stoev's construction, to generalize the construction by allowing birth and killing of L\'evy processes, and to obtain limit theorems in which Stoev's processes appear in a natural way as limits.

The paper is organized as follows. We start by describing a two-sided version of Stoev's construction in Section~\ref{subsec:levy_brown_resnick}. In Section~\ref{subsec:levy_brown_resnick_general} we generalize this construction to L\'evy processes with random birth and killing times. Stationary max-stable processes constructed in this way will be called \textit{L\'evy--Brown--Resnick} processes.  Mixed moving maxima representations of these processes will be constructed in Section~\ref{subsec:MMM} and some of their properties will be studied in Section~\ref{subsec:general_properties}. In Section~\ref{subsec:extremal_index} we compute the extremal index of a L\'evy--Brown--Resnick process in the case when the driving L\'evy process has no positive jumps. In Section~\ref{subsec:extremes_levy} we prove that the processes introduced by~\citet{sto2008} appear as limits of pointwise maxima of i.i.d.\ L\'evy processes, after applying suitable normalization procedures.   Finally, in Sections~\ref{subsec:extremes_levy_stable} and~\ref{subsec:extremes_OU} we generalize the original results of~\citet{bro1977} to the totally skewed $\alpha$-stable case.
The proofs are given in Sections~\ref{sec:proof_stationarity}, \ref{sec:proof_general_properties}, \ref{sec:proof_convergence}.

\begin{rem}
In this paper, we focus on continuous-time processes defined on $\R$. However, our results (except those of Sections~\ref{subsec:extremes_levy_stable} and~\ref{subsec:extremes_OU}) remain valid if we replace continuous time by discrete time and L\'evy processes by random walks.
\end{rem}

\subsection{L\'evy--Brown--Resnick processes}\label{subsec:levy_brown_resnick}
Let $\{L^+(t)\colon t\geq 0\}$ be a L\'evy process satisfying
\begin{equation}\label{eq:psi=0}
%\psi :=
\E \eee^{L^+(1)} = 1.
\end{equation}
Stoev~\cite{sto2008} showed that if $L_1^+,L_2^+,\ldots$ are i.i.d.\ copies of $L^+$ and, independently, $\sum_{i=1}^{\infty} \delta_{U_i}$ is a PPP on $\R$ with intensity $\eee^{-u}\dd u$, then the max-stable process
\begin{align}\label{LBR_1side}
  \eta(t) = \max_{i\in\N}\, (U_i + L^+_i(t)), \qquad t\geq 0,
\end{align}
is stationary on $\R_+$. Indeed, the mapping theorem for Poisson point processes implies, together with~\eqref{eq:psi=0}, that for any $t_0\geq 0$, the points $U_i+L^+_i(t_0)$, $i\in\N$, form a PPP with the same intensity $\eee^{-u}\dd u$. By the Markov property of the L\'evy processes $L_i^+$, the time-shifted  process $\{\eta(t_0+t)\colon t\geq 0\}$ has the same law as the original process $\{\eta(t)\colon t\geq 0\}$.

\begin{figure}[t]
  \includegraphics[width=1\textwidth, trim = 0mm 18mm 0mm 18mm]{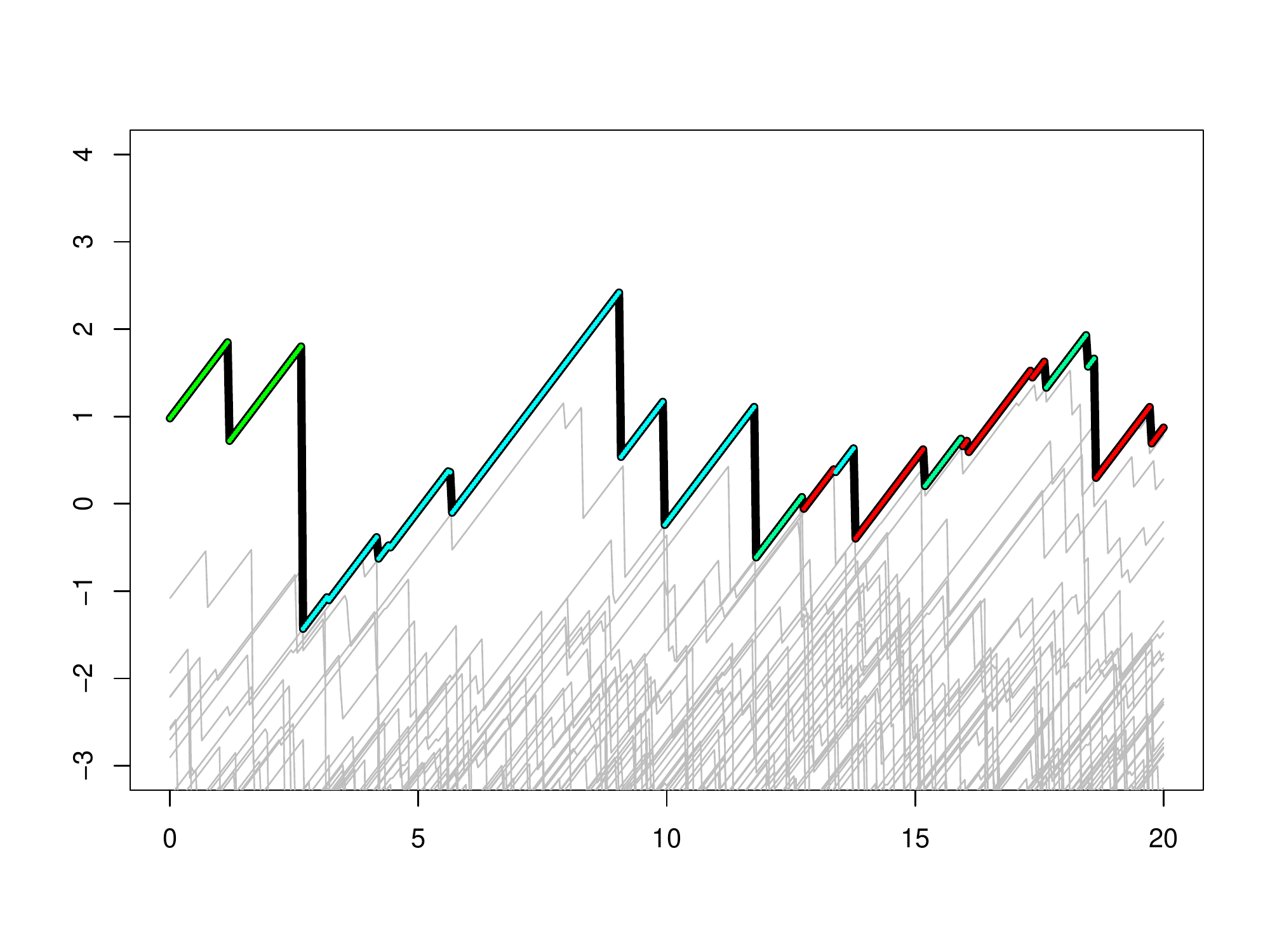}
   \caption{L\'evy--Brown--Resnick process generated by drifted compound Poisson processes with exponential jumps. Different colors indicate different particles contributing to the maximum process.}\label{fig:LBR_Poisson}
\end{figure}

How to obtain a two-sided stationary extension of the process $\eta$? To this end, let us adopt the particle system interpretation of the de Haan representation; see Section~\ref{subsec:intro}.  Take some positive time $T>0$ and look at some particle from the system conditioned to be in spatial position $x$ at time $T$. The conditional intensity of finding this particle in spatial position $y$ at time $T-t$ is $\eee^{-y} \dd y \, q_t^+(y,\dd x)/ (\eee^{-x} \dd x)$,
where $q^+_t(x,\dd y)$ is the probability transition kernel of the process $L^+$ (describing the forward in time motion of   particles).
That is, the probability transition kernel $q_t^{-}(x,\dd y)$ of the L\'evy process $L^-$ (which describes the backward in time motion of particles) is related to $q^+_t(x,\dd y)$  by the duality relation
\begin{equation}\label{eq:q_+_q_-_dual_0}
\eee^{-x} \dd x \cdot q_t^-(x,\dd y)  = \eee^{-y} \dd y \cdot q_t^+(y, \dd x).
\end{equation}
With other words, the process $-L^-$ can be obtained from $L^+$ by exponential tilting (Esscher transform):
\begin{equation}\label{eq:tilting}
\P[-L^-(t)\in B] = \E [\eee^{L^+(t)} \ind_{L^+(t)\in B}],
\end{equation}
for all Borel sets $B\subset \R$.  Note that $L^-$ satisfies $\E \eee^{L^-(t)}=1$, exactly as $L^+$.
Taking independent realizations of $L^+$ and $L^-$, we define the two-sided process
\begin{align}\label{twosided}
    L(t) =
    \begin{cases}
      L^+(t), & \quad t\geq 0,\\
      L^-(-t), & \quad t < 0.
    \end{cases}
\end{align}
\begin{theo}\label{Cor_main}
Let $\sum_{i=1}^{\infty} \delta_{U_i}$ be a PPP on $\R$ with intensity $\eee^{-u}\dd u$ and, independently, let $L_1,L_2,\ldots$ be i.i.d.\ copies of the process $\{L(t)\colon t\in\R\}$. Then, the process
\begin{align}\label{eta_def}
\eta(t) = \max_{i\in\N}\, (U_i +  L_{i}(t)) , \qquad t\in\R,
\end{align}
is max-stable and stationary.
\end{theo}
Theorem~\ref{Cor_main} is a particular case of a more general Theorem~\ref{thm1}, below.
\begin{ex}
%\noindent
\begin{itemize}
\item
  Let $\{B(t)\colon t\in\R\}$ be a two-sided standard Brownian motion. The one-sided process $L^+(t)=B(t)-t/2$, $t\geq 0$, satisfies~\eqref{eq:psi=0}. It follows from~\eqref{eq:tilting} that the dual process $L^-$ has the same law as $L^+$. Hence, the two-sided process $L$ can be identified with $B(t)-|t|/2$, and we recover the original Brown--Resnick process~\eqref{BR}.
\item
  Let $\{N^+(t)\colon  t\geq 0\}$ be a Poisson process with intensity $\lambda>0$. Then, the process $L^+(t) = N^+(t)-(\eee-1)\lambda t$ satisfies~\eqref{eq:psi=0}. The dual process is given by $L^-(t) = (\eee-1)\lambda t - N^-(t)$, where $\{N^-(t)\colon t\geq 0\}$ is a Poisson process with intensity $\eee \lambda$. The two-sided process $L$ is then
  $$
L(t) =
    \begin{cases}
      N^+(t)-(\eee-1)\lambda t, & \quad t\geq 0,\\
      - N^-(-t) - (\eee-1)\lambda t , & \quad t < 0.
    \end{cases}
$$
\item
  Generalizing the above examples, one can show that if the process $L^+$ has L\'evy triple $(\nu_+, \sigma^2_+, d_+)$, then the process $L^-$ has the L\'evy triple $(\nu_-, \sigma^2_-, d_-)$, where the variance $\sigma^2_+=\sigma^2_-$ is the same in both cases, and the L\'evy measures $\nu_+$ and $\nu_-$ are related by
  $$
  \nu_-(-\dd x) = \eee^{x} \nu_+(\dd x).
  $$
  The proof follows from~\eqref{eq:tilting} and the well-known behavior of the L\'evy triple under the Esscher transform; see~\cite[Theorem~3.9]{kyprianou_book}.  Note that the drifts $d_+$ and $d_-$ are uniquely determined by the remaining parameters and the relation~\eqref{eq:psi=0}.
\end{itemize}
\end{ex}

Figure~\eqref{fig:LBR_Poisson} shows a max-stable process $\eta$ generated by a compound Poisson process with exponential jump sizes, and the complete set of trajectories $\{U_i+L_i(\cdot)\colon i\in\N\}$.

\begin{rem}
Two-sided processes obtained by running a Markov process forward in time and the dual Markov process backward in time, as in~\eqref{twosided}, are well-known in probabilistic potential theory; see~\citet{mitro}.
\end{rem}

\begin{figure}[t]
  \includegraphics[width=1\textwidth, trim = 0mm 18mm 0mm 18mm]{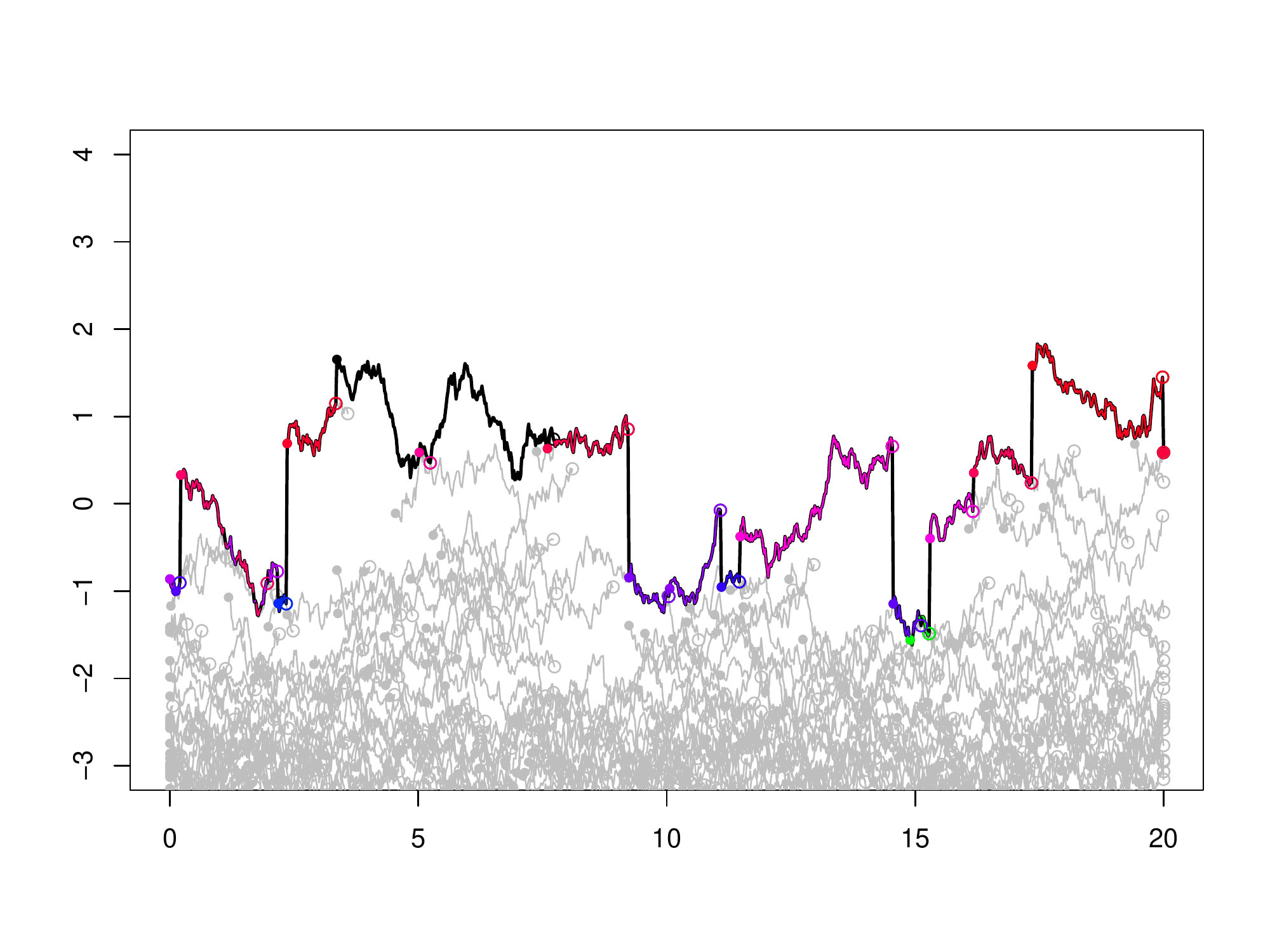}
   \caption{L\'evy--Brown--Resnick process generated by drifted Brownian motions; see Example~\ref{ex:LBR_BM}. Both birth and killing times are finite. Different colors indicate different particles contributing to the
maximum process.}\label{fig:LBR_BM}
\end{figure}

\subsection{Generalization to random creation and killing times}\label{subsec:levy_brown_resnick_general}
In this section we generalize the construction of L\'evy--Brown--Resnick processes to the case when~\eqref{eq:psi=0} is not satisfied.
We start with a L\'evy process $\{\xi(t)\colon t\geq 0\}$ for which
\begin{align}\label{psi}
  \psi(1) :=  \log \E \eee^{\xi(1)} < \infty.
\end{align}
We do \textit{not} require that $\psi(1)=0$. Additionally, we need two parameters $\theta_+\geq 0$ and $\theta_-\geq 0$ satisfying the relation
\begin{align}\label{rates}
    \psi(1) = \theta_- - \theta_+.
\end{align}

We will construct a stationary system of independent particles which move according to the law of the process $\xi$ and where $\theta_+$ and $\theta_-$ play the role of killing and birth rates, respectively. First, we describe the forward motion of particles, that is, we restrict ourselves to non-negative times $t\geq 0$.  Let $\pi_0=\sum_{i=1}^{\infty} \delta_{U_i}$ be a PPP on $\R$ having intensity $\eee^{-u}\dd u$. Consider a collection of particles starting at the points $U_i$ and moving independently of each other and of $\pi_0$ according to the law of the L\'evy process $\xi$. Then, at any time $t\geq 0$ the positions of the particles form a PPP with intensity $\eee^{\psi(1) t} \eee^{-u}\dd u$. This easily follows from the transformation theorem for the PPP. So, the intensity of the particles is not preserved except when $\psi(1)=0$. In order to obtain a \textit{stationary} particle system, it is natural to introduce creation (in the case $\psi(1)<0$) or killing (in the case $\psi(1)>0$) of particles. In fact, it is possible to consider both operations simultaneously. At any moment of time $t\geq 0$, let us kill any particle (independently of everything else) with rate $\theta_-$. Independently, at any moment of time $t\geq 0$, we create a new particle at spatial position $u\in\R$ with intensity $\theta_+ \eee^{-u}\dd u \dd t$. It is clear that the intensity of particles is preserved (meaning that it equals $\eee^{-u}\dd u$ at any time $t\geq 0$) if and only if the rates $\theta_+$ and $\theta_-$  satisfy~\eqref{rates}.

Thus, we constructed a one-sided stationary particle system defined for $t\geq 0$. In order to obtain a two-sided version of the system, note that when looking at the system backwards in time, creation of particles appears as killing and vice versa. This means that for $t\leq 0$, the rates $\theta_+$ and $\theta_-$ interchange their roles. That is, for $t\leq 0$,  $\theta_-$ is the creation rate, whereas $\theta_+$ is the killing rate.

Let us describe our construction in more precise terms. There are three types of particles in the system: those which are present at time $0$, those which are born after time $0$, and those which were killed before time $0$. Quantities related to the particles of the latter two types will be marked by a tilde.
We assume that:
\begin{enumerate}
\item [(A1)] The initial spatial positions of those particles which are present at time $0$ form a PPP $\pi_0=\sum_{i=1}^{\infty}\delta_{U_i}$ on $\R$ with intensity $\eee^{-u} \dd u$.
\item [(A2)] The times of birth and the initial positions of particles born after time $0$ form a PPP $\pi_+=\sum_{i=1}^{\infty}\delta_{(\tilde T_i^+, \tilde U_i^+)}$  on $(0,\infty)\times \R$ with intensity $\theta_+\dd t \times \eee^{-u} \dd u$.
\item [(A3)] The killing times and the terminal positions of particles killed before time $0$ form a PPP  $\pi_-=\sum_{i=1}^{\infty} \delta_{(\tilde T_i^-, \tilde U_i^-)}$  on $(-\infty,0)\times \R$ with intensity $\theta_- \dd t \times  \eee^{-u} \dd u $.
\end{enumerate}
We assume that after its birth every particle moves (forward in time) according to the law of the L\'evy process $\{L^+(t)\colon t\geq 0\}$ obtained from $\{\xi(t)\colon t\geq 0\}$ by killing it with rate $\theta_-$. That is, the subprobability transition kernel $q_t^+(x, \dd y)$ of $L^+$ is related to the probability transition kernel $p_t(x,\dd y)$ of $\xi$ by
\begin{equation}\label{eq:q_t^+_p_t}
q_t^+(x, \dd y) = \eee^{-\theta_- t} p_t(x, \dd y).
\end{equation}
Let also $\{L^-(t)\colon t\geq 0\}$ be the L\'evy process which is the dual of $L^+$ w.r.t.\ the  (in general, non-invariant) measure $\eee^{-u} \dd u$. That is, the subprobability transition kernel $q^{-}(x,\dd y)$ of $L^-$ is given by
\begin{equation}\label{eq:q_+_q_-_dual}
\eee^{-x} \dd x \cdot q_t^-(x,\dd y)  = \eee^{-y} \dd y \cdot q_t^+(y, \dd x).
\end{equation}
Note that $L^-$ may be killed after finite time, in general. From~\eqref{eq:q_+_q_-_dual} and~\eqref{rates} it follows easily that the killing rate of the process $L^-$ is $\theta_+$.
Consider a two-sided process $\{L(t)\colon t\in\R\}$ obtained by pasting together independent realizations of $L^+$ and $L^-$:
\begin{align}\label{twosided_1}
    L(t) =
    \begin{cases}
      L^+(t), & \quad t\geq 0,\\
      L^-(-t), & \quad t < 0.
    \end{cases}
\end{align}
Our assumptions on the motion of particles are as follows:
\begin{enumerate}
\item[(A4)] The motion of the particles which are present at time $0$ is given by i.i.d.\ copies $L_1,L_2,\ldots$ of the process $\{L(t)\colon t\in\R\}$.
\item[(A5)] The forward in  time motion of  particles which are born after time $0$ is given by i.i.d.\ copies $\tilde L_1^+,\tilde L_2^+,\ldots$ of the process $\{ L^+(t)\colon t\geq 0\}$.
\item[(A6)] The backward in time motion of  particles which were killed before time $0$ is given by i.i.d.\ copies $\tilde L_1^-,\tilde L_2^-,\ldots$ of the process $\{ L^-(t)\colon t\geq 0\}$.
\item[(A7)] The random elements $\pi_0, \pi_+, \pi_-$, $L_i, \tilde L_i^+, \tilde L_i^-$, $i\in\N$, are independent.
\end{enumerate}
The trajectories of particles which are present at time $t=0$ are given by the two-sided random functions
\begin{equation}
V_i(t) = U_i + L_i(t), \quad t\in\R.
\end{equation}
The trajectory of a particle which is born at time $\tilde T_i^+>0$ is given by the one-sided random function
\begin{equation}
\tilde V_i^+(t)
=
\begin{cases}
-\infty, &t<\tilde T_i^+,\\
\tilde U_i^+ + \tilde L_i^+(t-\tilde T_i^+), & t \geq \tilde T_i^+.
\end{cases}
\end{equation}
Similarly, the trajectory of a particle which was killed at time $\tilde T_i^-<0$ is given by the one-sided random function
\begin{equation}
\tilde V_i^-(t)
=
\begin{cases}
\tilde U_i^- + \tilde L_i^-(\tilde T_i^--t), & t < \tilde T_i^-,\\
-\infty, &t \geq \tilde T_i^-.
\end{cases}
\end{equation}
Note that killing of a particle is interpreted as changing its coordinate to $-\infty$. We always agree that $L^+$ should be right-continuous with left limits (c\`adl\`ag), whereas $L^-$ should be left continuous with right limits, so that $L$ is again c\`adl\`ag. We regard $V_i, \tilde V_i^+, \tilde V_i^-$ as elements of the Skorokhod space $\bar D$ of c\`adl\`ag functions defined on $\R$ and taking values in $\R\cup \{-\infty\}$. Define the shifts $T_t:\bar D\to \bar D$, $t\in\R$, by  $T_t f(s) = f(s-t)$.
The next result generalizes the L\'evy--Brown--Resnick processes constructed in Section~\ref{subsec:levy_brown_resnick} by allowing random birth and killing of spectral functions.

\begin{theo}\label{theo:inv_two_sided_PPP}
The law of the following PPP on $\bar D$ is invariant with respect to the time shifts $T_t$, $t\in\R$,
$$
\Pi:=\sum_{i=1}^{\infty} \delta_{V_i} + \sum_{i=1}^{\infty} \delta_{\tilde V_i^+} + \sum_{i=1}^{\infty} \delta_{\tilde V_i^-},
$$
 and its infinite intensity measure is
\begin{align}
\lefteqn{\mu_{\Pi} \{f\in \bar D\colon f(t_1)\in \dd x_1,\ldots, f(t_n) \in \dd x_n\}}\label{eq:intensity_kuznetsov}\\
&=
\eee^{-x_1} \dd x_1\cdot q_{t_2-t_1}^+(x_1, \dd x_2)\cdot \ldots \cdot q_{t_n-t_{n-1}}^+(x_{n-1}, \dd x_n),\notag
\end{align}
for all $t_1<\ldots<t_n$ and $x_1,\ldots,x_n\in\R$.
As a consequence, the process
\begin{equation}\label{eq:eta_with_killing_def}
  \eta(t) := \max\{V_i(t), \tilde V_i^+(t), \tilde V_i^-(t)  \colon i\in\N\}, \quad t\in\R,
\end{equation}
is max-stable and stationary.
\end{theo}
The proof of Theorem~\ref{theo:inv_two_sided_PPP} will be given in Section~\ref{sec:proof_stationarity}.
It relies on a more general result on stationary particle systems which is
not only valid for L\'evy processes but also for Markov processes that possess
an invariant measure.

\begin{rem}
By using the duality relation between $q^+_t$ and $q^-_t$, see~\eqref{eq:q_+_q_-_dual}, the right-hand side of~\eqref{eq:intensity_kuznetsov} can be rewritten in the following form:
$$
\eee^{-x_n} \dd x_n \cdot q_{t_n-t_{n-1}}^-(x_n, \dd x_{n-1})\cdot \ldots \cdot q_{t_2-t_{1}}^-(x_2, \dd x_1). $$
\end{rem}

\begin{definition}
The stationary max-stable process $\eta$ defined in~\eqref{eq:eta_with_killing_def} will be called a \textit{L\'evy--Brown--Resnick} process.
\end{definition}

\begin{ex}[See Figure~\ref{fig:LBR_BM}]\label{ex:LBR_BM}
Let $\{B(t)\colon  t\geq 0\}$ be a standard Brownian motion. Fix a scale parameter $\sigma>0$ and a drift $\lambda\in\R$. Let $\xi(t) = \sigma B(t) + \lambda t$, $t\geq 0$. Then, $\psi(1)=\lambda + \frac 12 \sigma^2$; see~\eqref{psi}. Fix a killing rate $\theta_-\geq \psi(1)$ and let $L^+$ be the process obtained by killing $\xi$ with rate $\theta_-$. A straightforward calculation using~\eqref{eq:q_t^+_p_t} and~\eqref{eq:q_+_q_-_dual} shows that the dual process $L^-$ has the same law as $\sigma B(t)- (\sigma^2 + \lambda) t$ killed at rate $\theta_+ := \theta_--\psi(1)\geq 0$.
Figure~\eqref{fig:LBR_BM} shows the corresponding max-stable process $\eta$ together with the particle trajectories.
In the case when $\sigma=1$, $\lambda=-\frac 12$ and $\theta_-=\theta_+=0$, we recover the original Brown--Resnick process $\eta_{\text{BR}}$; see~\eqref{BR}.
\end{ex}

Equation \eqref{eq:intensity_kuznetsov} states that the intensity of $\Pi$ is the so-called \textit{Kuznetsov measure} associated with the killed L\'evy process $L^+$ and the excessive measure $\mu(\dd u)=\eee^{-u}\dd u$. Kuznetsov measures can be associated with any Markov process and any excessive $\sigma$-finite measure $\mu$; see~\cite{kuznetsov}. The excessivity means that $P_t\mu\leq \mu$, where $P_t$ is the transition kernel of the Markov process. In our case, the excesssivity of $\mu$ w.r.t.\ the kernel $q_t^+$ follows from the inequality $\theta_-\geq \psi(1)$; see~\eqref{rates}. The existence of Kuznetsov measures was established in~\cite{kuznetsov} using Kolmogorov's extension theorem; see also the work of~\citet{getoor_glover} and~\citet{mitro} for alternative constructions.

\begin{rem}
The measure $\mu_\Pi$ is the so-called \textit{exponent measure} of the max-stable process $\eta$, that is for all $y_1,\ldots,y_n\in\R$,
\begin{equation}\label{eq:fidi_eta_exponent measure}
\P[\eta(t_1)\leq y_1,\ldots, \eta(t_n)\leq y_n]
 =
\exp\left( -\mu_{\Pi}\{f\in\bar D\colon f(t_i)>y_i \text{ for some } i\}\right).
\end{equation}
\end{rem}
\begin{rem}
Denote the L\'evy triple of $\xi$ by $(\nu_+, \sigma^2_+, d_+)$. Let us show that the laws of the L\'evy--Brown--Resnick processes~\eqref{eq:eta_with_killing_def} are in one-to-one correspondence with quintuples  $(\nu_+, \sigma^2_+, d_+, \theta_+,\theta_-)$ satisfying~\eqref{psi} and~\eqref{rates}. By construction, any such quintuple determines the law of  $\eta$ uniquely. Let us prove the converse.  The law of $\eta$ determines the exponent measure $\mu_\Pi$ uniquely; see~\eqref{eq:fidi_eta_exponent measure}. From~\eqref{eq:intensity_kuznetsov} with $n=2$ it follows that $\mu_\Pi$ determines the kernel $q^+_t$ and hence, by~\eqref{eq:q_t^+_p_t}, the law of $\xi(1)$ and the rate $\theta_-$ uniquely. By~\eqref{rates}, $\theta_+$ is also uniquely determined. So, the law of $\eta$ determines the quintuple $(\nu_+, \sigma^2_+, d_+, \theta_+,\theta_-)$ uniquely.
\end{rem}
\begin{prop}\label{prop:reversion}
If $\{\eta(t)\colon t\in\R\}$ is the L\'evy--Brown--Resnick process determined by the quintuple $(\nu_+, \sigma^2_+, d_+, \theta_+,\theta_-)$, then the reversed process $\{\eta(-t)\colon t\in\R\}$ is also a L\'evy--Brown--Resnick process with the quintuple $(\nu_-, \sigma^2_-, d_-, \theta_-,\theta_+)$, where $\nu_-(-\dd x) = \eee^x \nu_+(\dd x)$, $\sigma_-^2=\sigma_+^2$ and $d_-$ is uniquely determined by the remaining parameters.
\end{prop}
\begin{proof}
From the definition of the L\'evy--Brown--Resnick processes it follows that the reversed process $\eta(-t)$ has the same structure as $\eta(t)$, but the pairs $(\theta_+, L^+)$ and $(\theta_-, L^-)$ interchange their roles. The relation between the L\'evy triples of $L^+$ and $L^-$ follows from the well-known transformation properties of L\'evy triples under exponential tilting; see~\cite[Theorem~3.9]{kyprianou_book}.
\end{proof}
\begin{cor}
The process $\eta$ is reversible, that is, $\{\eta(t)\colon t\in\R\}$ has the same law as $\{\eta(-t)\colon t\in\R\}$, if and only if $\xi$ is a Brownian motion with linear drift and $\theta_-=\theta_+$. In particular, if there is no killing, then $\eta$ is reversible if and only if it is the original Brown--Resnick process $\eta_{\text{BR}}$.
\end{cor}
\begin{proof}
From Proposition~\ref{prop:reversion} we immediately obtain that for a reversible process $\eta$ we must have $\theta_+=\theta_-$ and $\nu_+=\nu_-=0$.
\end{proof}

\subsection{An explicit mixed moving maximum representation}\label{subsec:MMM}
The construction of L\'evy--Brown--Resnick processes given in Section~\ref{subsec:levy_brown_resnick_general} divides the spectral functions according to whether they are present (that is, not equal to $-\infty$) at time $t=0$ or not.
One may ask whether there is a more natural, translation invariant construction. A possible way to obtain such construction is to choose on any trajectory from $\Pi$ some ``reference point'' in a translation invariant way. In the case when $\psi(1)=\theta_+=\theta_-=0$, all paths from the PPP $\Pi$ are defined on the whole real axis (with birth at time $-\infty$ and killing at time $+\infty$). In this case, it is natural to  choose the maximum of the trajectory as the reference point. Following this approach, \citet{engelke_ivanovs} obtained an explicit representation of $\eta$ as a translation invariant mixed moving maximum process.

Here, we will give a translation invariant construction of $\Pi$ in the case when at least one rate $\theta_-, \theta_+$ is strictly positive. Let us assume that $\theta_->0$. This assumption means that the birth time of each path in $\Pi$ is finite and it is natural to consider the birth point as the reference point of the path. The following objects will be needed to describe an alternative construction of $\Pi$:
\begin{enumerate}
\item[(B1)] Let $\rho_+ := \sum_{i=1}^{\infty}\delta_{(S_i^+,V_i^+)}$ be a PPP on $\R\times \R$ with intensity $\theta_+ \dd s \times \eee^{-v}\dd v$.
\item[(B2)] Let $L_1^+,L_2^+,\ldots$ be i.i.d.\ copies of the killed L\'evy process $\{L^+(t)\colon t\geq 0\}$.
\item[(B3)] Let the random elements $\rho_+,L_1^+,L_2^+,\ldots$ be independent.
\end{enumerate}
Consider particles which are born at times $S_i^+$, have initial spatial positions $V_i$, and move (forward in time) according to the processes $L_i^+$. The trajectories of these particles are given by the one-sided functions
\begin{equation}
W_i^+(t)
=
\begin{cases}
-\infty, &t< S_i^+,\\
V_i^+ + L_i^+(t - S_i^+), & t \geq S_i^+.
\end{cases}
\end{equation}
\begin{theo}\label{theo:MMM_rep}
Let $\theta_->0$. Then, the PPP $\Pi$ from Theorem~\ref{theo:inv_two_sided_PPP} has the same intensity as the PPP
$$
\Pi' := \sum_{i=1}^{\infty} \delta_{W_i^+}.
$$
\end{theo}
\begin{proof}
Let us denote by $\mu_{\Pi'}$ the intensity of the PPP $\Pi'$ on the Skorokhod space $\bar D$. We will show that $\mu_{\Pi'}$ coincides with the intensity $\mu_{\Pi}$ in~\eqref{eq:intensity_kuznetsov}. Fix $t_1<\ldots <t_n$ and $x_1<\ldots<x_n$. Since a path $f\in\Pi'$ can be born at any point $(s,v)\in\R^2$ with intensity $\theta_+ \dd s \times \eee^{-v}\dd v$, we have
\begin{align}
\lefteqn{\mu_{\Pi'}\{f\in \bar D\colon f(t_1)\in \dd x_1, \ldots, f(t_n)\in \dd x_n\}}\label{eq:mu_Pi}\\
&=
\left(\int_{-\infty}^{t_1} \int_{-\infty}^{+\infty} q^+_{t_1-s} (v,\dd x_1) \theta_+ \eee^{-v} \dd v \dd s\right) q_{t_2-t_1}^+(x_1, \dd x_2) \ldots  q_{t_n-t_{n-1}}^+(x_{n-1}, \dd x_n).\notag
\end{align}
Note that
$$
\int_{-\infty}^{+\infty}  q^+_{t_1-s}(v, \dd x_1) \eee^{-v} \dd v = \eee^{(\psi(1) - \theta_-)(t_1-s)} \eee^{-x_1}\dd x_1.
$$
Hence, the double integral on the right-hand side of~\eqref{eq:mu_Pi} equals
$$
 \int_{-\infty}^{t_1} \theta_+ \eee^{(\psi(1) - \theta_-)(t_1-s)} \dd s\cdot  \eee^{-x_1} \dd x_1 = \eee^{-x_1}\dd x_1,
$$
 where  we used the basic relation~\eqref{rates}. The resulting expression for $\mu_{\Pi'}$ coincides with the formula for $\mu_{\Pi}$ given in~\eqref{eq:intensity_kuznetsov}.
\end{proof}
In the case $\theta_+>0$ (which means that the killing times of the particles are finite), there is a ``backward'' representation of $\Pi$ analogous to the ``forward'' representation stated in Theorem~\ref{theo:MMM_rep}. For $\theta_+>0$, the killing points of the paths $(S_i^-,V_i^-)$ form a PPP
%$\rho_- := \sum_{i=1}^{\infty}\delta_{(R_i,V_i^-)}$
on $\R\times \R$ with intensity $\theta_- \dd s \times \eee^{-v}\dd v$. Attaching to each point $(S_i^-,V_i^-)$ a copy of the process $L^-$ backward in time, we obtain a system of paths which has the same law as $\Pi$. In the case when both $\theta_+$ and $\theta_-$ are non-zero (meaning that both birth and killing times of the paths are finite), both representations (the forward one and the backward one) are valid.

\subsection{General properties of L\'evy--Brown--Resnick processes}\label{subsec:general_properties}
Let $\eta$ be a L\'evy--Brown--Resnick process as constructed in the previous sections.
\begin{prop}\label{prop_finite}
Fix a compact set $K\subset \mathbb{R}$. Then, the set
$$
J:=\{i\in\N\colon \exists t\in K \text{ such that } \eta(t) = V_i(t) \text{ or } \eta(t) = \tilde V_i^+(t) \text{ or } \eta(t) = \tilde V_i^-(t)\}
$$
is a.s.\ finite. That is, with probability $1$, only finitely many paths $V_i, \tilde V_i^+, \tilde V_i^-$ contribute to the process $\{\eta(t) \colon  t\in K\}$.
\end{prop}
The proof of Proposition~\ref{prop_finite} will be given in Section~\ref{subsec:proof_prop_finite}.
Since the pointwise maximum of finitely many c\`adl\`ag functions is again c\`adl\`ag,
the sample paths of the process $\eta$ are c\`adl\`ag with probability $1$.

A convenient measure of dependence for max-stable processes is the \textit{extremal correlation function} defined by
\begin{equation}
\rho(t) = 2 + \log \P[\eta(0)<0, \eta(t)<0]\in [0,1].
\end{equation}
\begin{prop}\label{prop:extremal_corr}
The extremal correlation function of $\eta$ is given by
\begin{equation}\label{eq:extremal_corr}
\rho(t) = \eee^{-\theta_+ t} - \eee^{-\theta_- t} \int_0^{\infty} \eee^{u} \P[\xi(t)>u]\dd u,   \quad t\geq 0.
\end{equation}
In particular, in the case $\theta_- = \theta_+ = 0$, we have
\begin{equation}\label{eq:rho_formula}
\rho(t) = \E  \min \{1, \eee^{L(t)}\}, \qquad t\in\R.
\end{equation}
\end{prop}
The proof of Proposition~\ref{prop:extremal_corr} will be given in Section~\ref{subsec:proof_extremal_corr}.
Note that the existence of a mixed moving maxima representation for $\eta$
as shown in~\citet{engelke_ivanovs} and Section~\ref{subsec:MMM} implies that
$\eta$ is mixing. According to~\cite[Thm.~3.4]{sto2008} the latter is also equivalent
to $\lim_{t\to+\infty} \rho(t) = 0$.

%% \begin{cor}
%% We have $\lim_{|t|\to\infty} \rho(t) = 0$ and the process $\eta$ is mixing.
%% \end{cor}
%% \begin{proof}
%% Note that $\rho(-t)=\rho(t)$. According to~\cite[Thm.~3.4]{sto2008} a process with extremal correlation function $\rho$ is mixing if and only if $\lim_{t\to+\infty} \rho(t) = 0$. We have $0\leq \rho(t)\leq \eee^{-\theta_+t}$.
%% So, if $\theta_+>0$, then $\lim_{t\to +\infty} \rho(t)=0$. If $\theta_->0$, we can use time reversion to obtain the same result. So, let $\theta_-=\theta_+=0$. It follows from $\E \eee^{L^+(t)} = 1$, $t\geq 0$, that $\E L^+(t)\in [-\infty, 0)$,  hence by Theorem 7.2 in \citet{kyprianou_book}, we have that
%% $\lim_{t\to+\infty} L^+(t) = -\infty$ a.s. By~\eqref{eq:rho_formula} and the dominated convergence theorem, we obtain that $\lim_{t\to+\infty} \rho(t) = 0$.
%% \end{proof}
%% In fact, the process $\eta$ is even a mixed moving maximum process (which implies mixing). The mixed moving maximum property follows from $\lim_{|t|\to\infty} L(t) = -\infty$ by repeating the proof of Theorem~14 from~\cite{kab2009}. Moreover, an explicit construction of the mixed moving maxima representation can be found in~\citet{engelke_ivanovs} and Section~\ref{subsec:MMM}.

\subsection{Extremal index in the spectrally negative case}\label{subsec:extremal_index}
An important quantity associated with a stationary max-stable process $\eta$ is its extremal index; see~\cite[p.~67]{lea1983}. By the max-stability of $\eta$, for every $T>0$ we can find $\Theta(T)>0$ such that
\begin{equation}\label{eq:Theta_T_def}
\P\left[\sup_{t\in [0,T]}\eta(t) - \log \Theta(T)\leq x\right] = \exp(-\eee^{-x}),
\quad
x\in\R.
\end{equation}
The \textit{extremal index} of $\eta$  is defined as the limit
\begin{equation}\label{eq:Theta_def}
\Theta:=\lim_{T\to \infty} \frac{\Theta(T)}{T}.
\end{equation}
In the next theorem we compute the extremal index of a L\'evy--Brown--Resnick process $\eta$ in the case when the driving L\'evy process $\xi$ is spectrally negative. Recall that $\xi$ is called spectrally negative if it has no positive jumps, or, equivalently, if the L\'evy measure of $\xi$ is concentrated on the negative half-axis. For a spectrally negative L\'evy process $\xi$, the function
$$
\psi(u) := \log \E \eee^{u\xi(1)}
$$
is finite for all $u\geq 0$; see~\cite[Chapter VII]{ber1996}.  Let $\psi^{-1}(0)$ be the largest solution of $\psi(u)=0$.  The function $\psi:[\psi^{-1}(0), \infty)\to [0,\infty)$ is strictly increasing and continuous, and the inverse function is denoted by $\psi^{-1}$.
\begin{theo}\label{theo:extremal_index}
Let $\eta$ be a L\'evy--Brown--Resnick process generated by a L\'evy process $\xi$ that has no positive jumps. Then, the extremal index of $\eta$ is given by
\begin{equation}\label{eq:Theta}
\Theta =
\begin{cases}
\psi'(1), &\text{if } \theta_+=0,\\
\frac{\psi^{-1}(\theta_-)}{\psi^{-1}(\theta_-)-1}\theta_+, &\text{if } \theta_+>0.
\end{cases}
\end{equation}
\end{theo}
The proof of Theorem~\ref{theo:extremal_index} will be given in Section~\ref{subsec:proof_extremal_index}.

\subsection{Extremes of independent L\'evy processes}\label{subsec:extremes_levy}
The original Brown--Resnick process $\eta_{\text{BR}}$, see~\eqref{BR}, appeared as a limit of pointwise maxima of independent Brownian motions, after appropriate normalization. Let $B_1,B_2,\ldots$ be i.i.d.\ standard Brownian motions. Let $u_n$ be any sequence such that $1-\Phi(u_n)\sim 1/n$, where $\Phi$ is the standard normal distribution function. \citet{bro1977} proved that weakly on the space $C(\R)$,
\begin{equation}\label{eq:brown_resnick_res1}
\left\{\sqrt{2\log n} \left(\max_{i=1,\ldots,n} B_i\left(1+ \frac{t}{2\log n}\right) - u_n\right)\colon t\in \R\right\}
\toweak
\left\{\eta_{\text{BR}}(t) + \frac t2\colon t\in \R\right\}.
\end{equation}
To make the left-hand side of~\eqref{eq:brown_resnick_res1} defined for every $t\in\R$, we can extend $B_i$ to the negative half-axis in an arbitrary way, for example by requiring that $B_i(s)=0$ for $s<0$. The space $C(\R)$ is endowed with the topology of uniform convergence on compact intervals so that the weak convergence on $C(\R)$ is equivalent to the weak convergence on $C[-T,T]$ for every $T\geq 0$. See also \cite{kab2009, eng2014a} for other classes of processes whose maxima converge to $\eta_{\text{BR}}$.

By using the self-similarity of the Brownian motion, we obtain that weakly on $C(\R)$,
\begin{equation}\label{eq:brown_resnick_res2}
\left\{ \max_{i=1,\ldots,n} B_i (2\log n + t) - u_n \sqrt{2\log n}\colon t\in\R\right\}
\toweak
\left\{\eta_{\text{BR}}(t) + \frac t2\colon t\in\R\right\}.
\end{equation}
Our aim is to generalize~\eqref{eq:brown_resnick_res2} to L\'evy processes.
Suppose that $\xi_1,\xi_2,\ldots$ are independent copies of a non-deterministic L\'evy process $\{\xi(t)\colon  t\geq 0\}$ such that the distribution of $\xi(1)$ is non-lattice and
\begin{align}\label{psi_1}
\psi(u) := \log \E \eee^{u\xi(1)} < \infty, \text{ for all } u \in [0,u_\infty),
\end{align}
where $u_\infty \in (0,+\infty]$ is maximal with this property. Let $s_1,s_2,\ldots$ be a sequence of non-negative real numbers such that
\begin{align}\label{growth}
\lambda := \lim_{n\to\infty} \frac{\log n}{s_n}  \in (0, \infty).
\end{align}
We are interested in the functional limit behavior of the process
\begin{align*}
  M_n(t) := \max_{i=1,\dots, n} \xi_i(s_n+t).
\end{align*}
To state our limit theorem on $M_n(t)$, we need to introduce some notation. Note that $\psi(0)=0$ and that the function $\psi'$ is a strictly increasing and infinitely differentiable bijection between $(0,u_\infty)$ and $(\beta_0,\beta_{\infty})$, where
$$
\beta_0=\lim_{u\downarrow 0}\psi'(u)=\E \xi(1)\in \R\cup\{-\infty\},
\quad
\beta_{\infty} = \lim_{u\uparrow u_\infty}\psi'(u) \in \R\cup\{+\infty\}.
$$
The information function  $I$ is defined as the Legendre--Fenchel transform
of $\psi$, that is
\begin{equation}\label{eq:I_def}
I(\psi'(u)) = u\psi'(u) - \psi(u), \quad u \in (0,u_\infty).
\end{equation}
Since every $\beta\in (\beta_0,\beta_\infty)$ can be represented as $\beta =\psi'(u)$ for some $u\in (0,u_\infty)$, the function $I$ is defined on the interval $(\beta_0,\beta_\infty)$. Let $\lambda_\infty=\lim_{\beta\uparrow \beta_\infty} I(\beta)$, so that $I$ is a bijection between $(\beta_0,\beta_{\infty})$ and $(0,\lambda_\infty)$. Suppose additionally that $\lambda\in (0,\lambda_\infty)$ and denote by $\theta\in (0, u_\infty)$ the unique solution to $I(\psi'(\theta)) = \lambda$. Define a normalizing sequence $b_n$ by
\begin{equation}\label{eq:b_n_def}
b_n = I^{-1} (\lambda_n ) s_n \sim \psi'(\theta) s_n
\text{ with }
\lambda_n := \frac{\log n - \log(\theta\sqrt{2\pi \psi''(\theta)s_n})}{s_n}\ton \lambda.
\end{equation}
Let $L^+$ be the L\'evy process defined by $L^+(t) = \theta \xi(t) - \psi(\theta) t$, $t\geq 0$. Note that $L^+$ satisfies~\eqref{eq:psi=0}. Let $L$  be the corresponding two-sided process as in~\eqref{twosided} and~\eqref{eq:tilting}.
\begin{theo}\label{theo:BR_for_levy}
We have the following weak convergence of stochastic processes on the Skorokhod space $D(\R)$:
\begin{align}\label{conv}
\left\{ \max_{i=1,\ldots,n} \xi_i (s_n + t) - b_n \colon t\in\R\right\}
\toweak
\left\{\frac 1 \theta \eta(t) + \frac{\psi(\theta)}{\theta} t \colon t\in\R\right\}.
\end{align}
where $\eta$ is the L\'evy--Brown--Resnick process  corresponding
to $L$; see~\eqref{eta_def}.
\end{theo}
In order to make the left-hand side of~\eqref{conv} well-defined for all $t\in\R$, we define, say, $\xi_i(s)=0$ for $s<0$. The proof of Theorem~\ref{theo:BR_for_levy} will be given in Section~\ref{subsec:proo_brown_resnick_levy}. The Skorokhod space is endowed with the usual $J_1$-metric; see~\cite[Section~16]{billingsley_book}.
Restricting Theorem~\ref{theo:BR_for_levy} to $t=0$ we recover a  known result due to~\citet{ivchenko} and~\citet{durrett}:
\begin{equation}\label{eq:durrett_ivchenko}
\max_{i=1,\ldots,n} \xi_i (s_n) - b_n \todistr \exp(-\eee^{-\theta x}).
\end{equation}
Theorem~\ref{theo:BR_for_levy} is a functional version of~\eqref{eq:durrett_ivchenko}. Functional limit theorems for sums of geometric L\'evy processes of the form $\eee^{\beta \xi_i(s_n+t)}$ were obtained in~\cite{kabluchko_FCLT} with limits being certain stationary stable or Gaussian processes. Theorem~\ref{theo:BR_for_levy} can be viewed as the limiting case of the results of~\cite{kabluchko_FCLT} as $\beta\to +\infty$.

\begin{rem}
The lattice assumption on $\xi(1)$ cannot be removed. If $\xi(1)$ is lattice, then Theorem~\ref{theo:BR_for_levy} breaks down and instead we have weak convergence along certain subsequences of $n$ to a topological circle of limiting processes as in~\cite{lifshits} or~\cite{komlos_tusnady}.
\end{rem}

\subsection{Extremes of independent totally skewed \texorpdfstring{$\alpha$}{alpha}-stable  L\'evy processes}\label{subsec:extremes_levy_stable}
In this section we will generalize the results of ~\citet{bro1977} to totally skewed $\alpha$-stable L\'evy processes.
To this end, we will combine Theorem~\ref{theo:BR_for_levy} with the scaling property of these processes. Let us first recall some definitions related to $\alpha$-stable processes (cf.\ \cite{sam1994}).
A real-valued random variable $X$ is said to have an
$\alpha$-stable distribution $S_\alpha(\sigma,\beta,\mu)$ with parameters $\alpha\in (0,2]$, $\sigma \geq 0$, $\beta \in [-1,1]$
and $\mu\in \R$ if its characteristic function has the form
\begin{align*}
\E \exp(i\theta X)
=
\begin{cases}
\exp\left\{ -\sigma^\alpha |\theta|^\alpha(1 - i\beta \sign(\theta) \tan (\pi\alpha/2)) + i\mu\theta\right\}, &\alpha\neq 1,\\
\exp\left\{ -\sigma |\theta|(1 + \frac {2} \pi i \beta \sign(\theta)  \log |\theta|) + i\mu\theta\right\}, &\alpha=1,
\end{cases}
\end{align*}
for all $\theta \in \R$. In general, $\alpha$-stable distributions possess heavy power-law tails and are thus in the max-domain of attraction of the Fr\'echet (rather than Gumbel) distribution. An exception, which we will focus on, is the case of $\alpha$-stable random
variables that are totally skewed to the left, that is, $\beta = -1$.

Let $X$ be a random variable with distribution $S_\alpha(1,-1,0)$. It is known that in the case $\alpha\in [1,2]$, $X$ has positive density on the whole real line, whereas in the  case $\alpha\in (0,1)$ the density is concentrated on  the negative half-line. Asymptotic formulas for the right tail of $X$ near its right endpoint $x^*$ (which is $+\infty$ for $\alpha\in[1,2]$ an $0$ for $\alpha\in (0,1)$) are well-known; see~\cite{alb1993} or~\cite[Eq. 1.2.11]{sam1994}. For $\alpha\neq 1$ the tail asymptotics has the form
\begin{equation}\label{eq:asympt_tail_stable}
  \P[X>x] \sim  A_\alpha x^{-\frac{\alpha}{2(\alpha-1)}} \exp\{ -B_\alpha x^{\frac{\alpha}{\alpha-1}} \},
  \quad x\uparrow x^*,
\end{equation}
with certain explicit constants $A_{\alpha}$ and $B_{\alpha}$. Suppose now that $X_1,X_2,\ldots$, are i.i.d.\ copies of $X\sim S_{\alpha}(1,-1,0)$, where $\alpha\in (0,2]$.
Using~\eqref{eq:asympt_tail_stable} and standard asymptotic calculations, see Theorem 3.3.26 in~\cite{emb1997}, one can obtain that there is a sequence $b_{n,\alpha}$ (see~\eqref{eq:b_n_alpha_def}, below) and a number $\theta_\alpha>0$ (see~\eqref{eq:theta_alpha}, below) such that
\begin{equation}\label{conv_max}
(\log n)^{\frac 1 \alpha} \max_{i=1,\dots,n} X_i - b_{n,\alpha}
\todistr
\exp(-\eee^{-\theta_\alpha x}).
\end{equation}

We will obtain a functional version of~\eqref{conv_max}. For $\alpha\in (0,2]$ consider a L\'evy process $\{\xi_{\alpha}(t)\colon t\geq 0\}$ such that the distribution of $\xi_\alpha(t)$ is $S_\alpha(t,-1,0)$.
It is well known, see Proposition~1.2.12 in~\cite{sam1994}, that for $u\geq 0$ we have
$$
\psi_{\alpha}(u):= \log \E \eee^{u\xi_\alpha(1)} =
\begin{cases}
c_{\alpha} u^{\alpha}, &\alpha\neq 1,\\
c_{1} u\log u, &\alpha=1,
\end{cases}
\text{ with }
c_{\alpha}
=
\begin{cases}
-\frac 1{\cos \frac{\alpha \pi}2}, &\alpha\neq 1,\\
\frac{2}{\pi}, &\alpha=1.
\end{cases}
$$
Note that $c_{\alpha}>0$ for $\alpha\in [1,2]$, while $c_{\alpha}<0$ for $\alpha\in (0,1)$.
Let us apply Theorem~\ref{theo:BR_for_levy} to $\xi_\alpha$.
A straightforward computation yields that the information function $I=I_{\alpha}$ from~\eqref{eq:I_def} is given by
$$
I_{\alpha}(\beta)=
\begin{cases}
(\alpha-1) \left(\frac{\beta}{\alpha^{\alpha} c_{\alpha}}\right)^{\frac 1 {\alpha-1}} \beta, &\alpha\neq 1,\\
c_1 \eee^{\frac{\beta}{c_1}-1}, &\alpha=1,
\end{cases}
$$
where the interval $(\beta_0,\beta_\infty)$ on which $I_{\alpha}$ is defined is $(0,+\infty)$ in the case $\alpha\in (1,2]$,  $(-\infty,0)$ in the case $\alpha\in (0,1)$, and $(-\infty,+\infty)$ in the case $\alpha=1$.
Take $s_n=\log n$ (so that $\lambda=1$). We easily compute that the solution to $I_{\alpha}(\psi'_{\alpha}(\theta_\alpha))=1$ is given by
\begin{equation}\label{eq:theta_alpha}
\theta_\alpha =
\begin{cases}
((\alpha-1) c_{\alpha})^{-\frac 1\alpha}, &\alpha\neq 1,\\
\frac \pi 2,&\alpha=1.
\end{cases}
\end{equation}
Applying the Taylor expansion of $I_{\alpha}^{-1}$ to~\eqref{eq:b_n_def} and discarding the $o(1)$ terms, we obtain that the normalizing sequence $b_n=b_{n,\alpha}$ is given by
\begin{equation}\label{eq:b_n_alpha_def}
b_{n,\alpha}
=
\begin{cases}
\frac 1 {\theta_\alpha} \left(\frac{\alpha}{\alpha-1} \log n - \frac 12 \log (2\pi \alpha \log n)\right), &\alpha\neq 1,\\
 \left(\frac 2 \pi \log \frac {\pi\eee} {2}\right) \log n -  \frac 1 {\pi} \log (2\pi \log n), &\alpha=1.
\end{cases}
\end{equation}

Let now $\xi_{1,\alpha},\xi_{2,\alpha},\ldots$ be i.i.d.\ copies of $\xi_\alpha$. Applying Theorem~\ref{theo:BR_for_levy} we obtain that weakly on the Skorokhod space $D(\R)$ it holds that
\begin{equation}\label{eq:brown_resnick_stable_Levy}
\left\{\max_{i=1,\ldots,n} \xi_{i,\alpha} \left(\log n +  t\right) - b_{n,\alpha} \colon t\in\R\right\}
\toweak
\left\{\frac 1 {\theta_{\alpha}} \eta_{\alpha}(t) + \frac{\psi(\theta_\alpha)}{\theta_\alpha} t \colon t\in\R\right\},
\end{equation}
where $\eta_{\alpha}$ is a L\'evy--Brown--Resnick process defined as in Section~\ref{subsec:levy_brown_resnick} with
\begin{equation}\label{eq:L^+_alpha}
L^+(t) = \theta_{\alpha} \xi_{\alpha}(t) - \psi(\theta_{\alpha}) t,\;\;\; t\geq 0.
\end{equation}
Note that in the case $\alpha=2$, with $\xi_{i,\alpha}(t)= B_i(2t)$, we recover Brown and Resnick's result~\eqref{eq:brown_resnick_res2}.

Using the self-similarity of $\xi_\alpha$ we can also generalize~\eqref{eq:brown_resnick_res1}. Let us denote the limiting process in~\eqref{eq:brown_resnick_stable_Levy} by $\tilde \eta_\alpha$:
\begin{equation}\label{eq:tilde_eta_alpha}
\tilde \eta_{\alpha}(t)
:=
\frac 1 {\theta_\alpha} \eta_\alpha(t) + \frac{\psi(\theta_\alpha)}{\theta_\alpha} t.
\end{equation}
\begin{theo}\label{theo:BR_for_alpha_stable_levy}
For $\alpha\neq 1$, we have the following weak convergence of stochastic processes on the Skorokhod space $D(\R)$:
\begin{equation}\label{eq:brown_resnick_stable_near1}
\left\{ (\log n)^{\frac 1\alpha}\max_{i=1,\ldots,n} \xi_{i,\alpha} \left(1 +  \frac{t}{\log n}\right) - b_{n,\alpha} \colon t\in\R\right\}
\toweak
\left\{\tilde \eta_{\alpha}(t) \colon t\in\R\right\}.
\end{equation}
For $\alpha=1$ we have, weakly on $D(\R)$,
\begin{equation}\label{eq:brown_resnick_stable_near1_alpha1}
\left\{ \log n \max_{i=1,\ldots,n} \xi_{i,1} \left(1 +  \frac{t}{\log n}\right) - \tilde b_{n,1}(t) \colon t\in\R\right\}
\toweak
\left\{\tilde \eta_{1}(t) \colon t\in\R\right\},
\end{equation}
where $\tilde b_{n,1}(t)=b_{n,1} + \frac 2 \pi (\log n) (\log\log n) t$.
\end{theo}
\begin{proof}
It is well known~\cite[Section 3.1]{sam1994} that for $\alpha\neq 1$, the process $\xi_{\alpha}$ is $1/\alpha$-self-similar, that is  for all $c>0$,
\begin{equation}\label{eq:self_similar_alpha_neq_1}
\{\xi_{\alpha}(ct)\colon t\geq 0\} \eqd \{c^{1/\alpha}\xi_{\alpha}(t)\colon t\geq 0\}.
\end{equation}
Combining~\eqref{eq:brown_resnick_stable_Levy} with the self-similarity, we obtain~\eqref{eq:brown_resnick_stable_near1}.
In the case $\alpha=1$ the self-similarity breaks down and instead one has
\begin{equation}\label{eq:self_similar_alpha_1}
\{\xi_{1}(ct)\colon t\geq 0\} \eqd \left\{c \xi_1(t) - \frac {2}{\pi} (c\log c) t\colon t\geq 0\right\}.
\end{equation}
Combining~\eqref{eq:self_similar_alpha_1} and~\eqref{eq:brown_resnick_stable_Levy}, we obtain~\eqref{eq:brown_resnick_stable_near1_alpha1}.
\end{proof}

\subsection{Extremes of independent totally skewed \texorpdfstring{$\alpha$}{alpha}-stable Ornstein--Uhlenbeck processes}\label{subsec:extremes_OU}
In addition to their result~\eqref{eq:brown_resnick_res1} on extremes of i.i.d.\ Brownian motions, \citet{bro1977} proved a similar result for Ornstein--Uhlenbeck processes. Let $Y_1,Y_2,\ldots$ be i.i.d.\ copies of the stationary Gaussian Ornstein--Uhlenbeck process
$$
Z(t):=\eee^{-t/2} B(\eee^{t}), \quad t\in\R.
$$
Then, with $u_n$ satisfying $1-\Phi(u_n)\sim 1/n$,  \citet{bro1977} proved that
\begin{equation}\label{eq:brown_resnick_OU}
\left\{\sqrt{2\log n} \left(\max_{i=1,\ldots,n} Z_i \left(1 + \frac t{2\log n}\right) - u_n\right) \colon t\in\R\right\}
\toweak
\left\{\eta_{\text{BR}}(t)\colon t\in\R\right\}
\end{equation}
weakly on $C(\R)$.
We now establish a generalization of this result in the totally skewed $\alpha$-stable case. As in the previous section, let $\{\xi_\alpha(t)\colon t\geq 0\}$ be a L\'evy process with $\xi_\alpha(t)\sim S_{\alpha}(t,-1,0)$, where $\alpha\in (0,2]$.
The associated Ornstein--Uhlenbeck process $\{Z_\alpha(t)\colon t\in\R\}$ is defined by
\begin{equation}\label{eq:Y_alpha_def}
  Z_\alpha(t) =
  \begin{cases}
  \eee^{-t/\alpha} \xi_\alpha(\eee^t), &\alpha\neq 1,\\
  \eee^{-t}\xi_1(\eee^t) + \frac{2}{\pi} t, &\alpha=1.
  \end{cases}
\end{equation}
The self-similarity of $\xi_\alpha$ (or~\eqref{eq:self_similar_alpha_1} in the case $\alpha=1$) implies that the process $Z_\alpha$ is stationary with $S_\alpha(1,-1,0)$ margins. Let $Z_{1,\alpha},Z_{2,\alpha},\ldots$ be i.i.d.\ copies of $Z_\alpha$.
\begin{theo}\label{theo:BR_for_alpha_stable_OU}
For $\alpha\neq 1$, we have the following weak convergence of stochastic processes on the Skorokhod space $D(\R)$:
\begin{equation}\label{eq:brown_resnick_OU_stable_alpha_neq 1}
\left\{ (\log n)^{\frac 1 \alpha} \max_{i = 1,\ldots,n} Z_{i,\alpha}\left(\frac{t}{\log n}\right) -b_{n,\alpha}\colon t\in\R\right\}
\toweak \left\{\frac{1}{\theta_\alpha}\eta_{\alpha}(t)\colon t\in \R\right\},
\end{equation}
where $\eta_{\alpha}$ is a L\'evy--Brown--Resnick process defined as in Section~\ref{subsec:levy_brown_resnick} with $L^+$ as in~\eqref{eq:L^+_alpha}.
For $\alpha=1$ the result takes the form
\begin{equation}\label{conv_max_W}
\left\{ (\log n) \max_{i = 1,\ldots,n} Z_{i,1}\left(\frac{t}{\log n}\right) - \tilde b_{n,1}(t)\colon t\in\R\right\}
\toweak \left\{\frac{2}{\pi}\eta_{1}(t)\colon t\in \R\right\}.
\end{equation}
\end{theo}
Theorem~\ref{theo:BR_for_alpha_stable_OU} will be deduced from Theorem~\ref{theo:BR_for_alpha_stable_levy} using~\eqref{eq:Y_alpha_def}. The proof will be given in Section~\ref{subsec:proof_OU}. The study of the pointwise  maximum of many independent stochastic processes over an infinitesimal interval (see Theorems~\ref{theo:BR_for_alpha_stable_levy} and~\ref{theo:BR_for_alpha_stable_OU}) is closely related to the results, due to Albin~\cite{alb1993}, \cite{alb97}, \cite{alb1998}, on the maximum of a single  totally-skewed $\alpha$-stable process over a finite or increasing interval. The drifted process $L^+$, see~\eqref{eq:L^+_alpha}, appeared in the works of Albin as an extremal tangent process describing the behavior of a totally skewed $\alpha$-stable process after reaching a high level.

\section{Proofs: Stationarity results}\label{sec:proof_stationarity}
\subsection{Stationary systems of independent Markov processes}
Let $(E,d_E)$ be a Polish space with metric $d_E$ and Borel $\sigma$-algebra $\mathcal{E}$.
Let $\{P_t^+\colon t\geq 0\}$ and $\{P_t^-\colon t\geq 0\}$ be two Markov probability transition semigroups on $E$ which are in duality w.r.t.\ some locally finite measure $\mu$. This means that
\begin{align}\label{duality}
  \int_A P_t^+(x, B) \mu(\dd x) = \int_B P_t^-(y, A) \mu(\dd y), \quad \text{for all } A,B\in\mathcal E.
\end{align}
In particular, the measure $\mu$ is invariant w.r.t.\ both $P_t^+$ and $P_t^-$:
\begin{align*}
\int_E P_t^+(x, B) \mu(\dd x) = \mu(B) = \int_E P_t^-(x, B) \mu(\dd x), \quad \text{for all } B\in\mathcal E.
\end{align*}

Consider a system of particles located in $E$ and moving independently of each other according to the following rules. The positions of particles at time $0$ form a Poisson point process (PPP) $\pi_0:=\sum_{i} \delta_{U_i}$ with
intensity measure $\mu$. The motion of particles is described as follows. For each $i\in\N$ consider Markov processes $\{\xi_i^+(t)\colon t\geq 0\}$ and $\{\xi_i^-(t)\colon t\geq 0\}$ which both start at $U_i$ and have transition semigroups $P_t^+$ and $P_t^-$, respectively. We assume that the processes $\xi_i^+, \xi_i^-$, $i\in\N$, are conditionally independent given $\pi_0$. Then, the position of particle $i\in\N$ at time $t\in\R$ is given by the two-sided process
\begin{align*}
  \xi_i(t) =
  \begin{cases}
    \xi_i^+(t), & \quad t\geq 0,\\
    \xi_i^-(-t), & \quad t < 0.
  \end{cases}
\end{align*}
We assume that the sample paths of the Markov process $\{\xi_i^+(t)\colon t\geq 0\}$ are right-continuous with left limits (c\`adl\`ag), whereas the sample paths of $\{\xi_i^-(t)\colon t\geq 0\}$ are left-continuous with right limits. Then, the sample paths of $\xi_i$ are c\`adl\`ag.

The positions of the particles at time $t\in\R$ are given by the point process $\pi_t:=\sum_{i} \delta_{\xi_i(t)}$ (which is a PPP on $E$), whereas the complete evolution of the particle system can be encoded by the point process $\Pi=\sum_{i} \delta_{\xi_i}$ (which is a Poisson point process on $D(\R, E)$, the Skorokhod space of
c\`adl\`ag functions from $\R$ to $E$).
Denote by $T_t:D(\R, E) \to D(\R, E)$ the shift operators
given by $T_tf(s) = f(s-t)$ with $f\in D(\R, E)$ and $s, t\in \R$.

\begin{theo}\label{thm1}
With the notation from above, the Poisson point process $\Pi=\sum_{i\in\N} \delta_{\xi_i}$
is stationary, that is for any $t\in\R$,
  \begin{align*}
    \sum_{i} \delta_{\xi_i} \eqdistr \sum_{i} \delta_{T_t \xi_i}.
  \end{align*}
\end{theo}
\begin{proof}
At least in the one-sided case the result is well known, see~\cite{brown} and the references therein, but we give a short proof for completeness. Fix some times $t_1\leq \dots\leq t_m \leq 0 \leq t_{m+1} \leq \dots \leq t_{m+n}$.
The intensity measure of $\Pi$ is given by
\begin{align*}
\lefteqn{\mu_{\Pi}\{f\in D(\R,E)\colon f(t_1)\in\dd x_1,\ldots, f(t_{n+m})\in\dd x_{m+n}\}}\\
&=
\int_E \mu(\dd x)
P_{-t_m}^-(x, \dd x_{m}) \ldots P_{t_2-t_1}^-(x_2, \dd x_{1})P_{t_{m+1}}^+(x, \dd x_{m+1}) \ldots P_{t_{m+n}-t_{m+n-1}}^+(x_{m+n-1}, \dd x_{m+n}).
\end{align*}
Repeatedly using identity~\eqref{duality} to replace $\mu(\dd x) P^-_t(x,\dd y)$ by $P^+_t(y,\dd x)\mu(\dd y)$, we obtain
\begin{align*}
\lefteqn{\mu_{\Pi}\{f\in D(\R,E)\colon f(t_1)\in\dd x_1,\ldots, f(t_{n+m})\in\dd x_{m+n}\}}\\
&=
\int_E \mu(\dd x_1)
P_{t_2-t_1}^+(x_1, \dd x_{2}) \ldots  P_{-t_m}^+(x_m, \dd x) P_{t_{m+1}}^+(x, \dd x_{m+1}) \ldots P_{t_{m+n}-t_{m+n-1}}^+(x_{m+n-1}, \dd x_{m+n})\\
&=P_{t_2-t_1}^+(x_1, \dd x_{2}) \ldots  P_{t_{m+1}-t_m}^+(x_m, \dd x_{m+1}) \ldots P_{t_{m+n}-t_{m+n-1}}^+(x_{m+n-1}, \dd x_{m+n}),
\end{align*}
where in the last step we performed integration over $\dd x$ and used the formula
$$
\int_{E} P_{-t_m}^+(x_m, \dd x) P_{t_{m+1}}^+(x, \dd x_{m+1})=P^+_{t_{m+1} - t_m}(x_m, \dd x_{m+1}).
$$
Clearly, the resulting expression for $\mu_\Pi$ does not change if we increase all $t_i$'s by the same value.
\end{proof}

\subsection{Proof of Theorem~\ref{theo:inv_two_sided_PPP}}
In order to encode a motion of a particle which has random birth and death times, it is convenient to introduce an extended state space $E = \R^2$; see Figure~\ref{fig:particles}. For a point $(x,s)\in \R^2$, the first coordinate $x$ is the usual spatial position of the particle.  The second coordinate $s$ indicates whether the particle is not yet born
($s<0$, in which case $|s|$ is the time remaining to the birth), it is alive ($s=0$), or it has already been killed ($s>0$, in which case $s$ is the time elapsed after the killing event), respectively.

\begin{figure}[t]
  \begin{centering}
  \includegraphics[width=0.5\textwidth, height=0.5\textwidth]{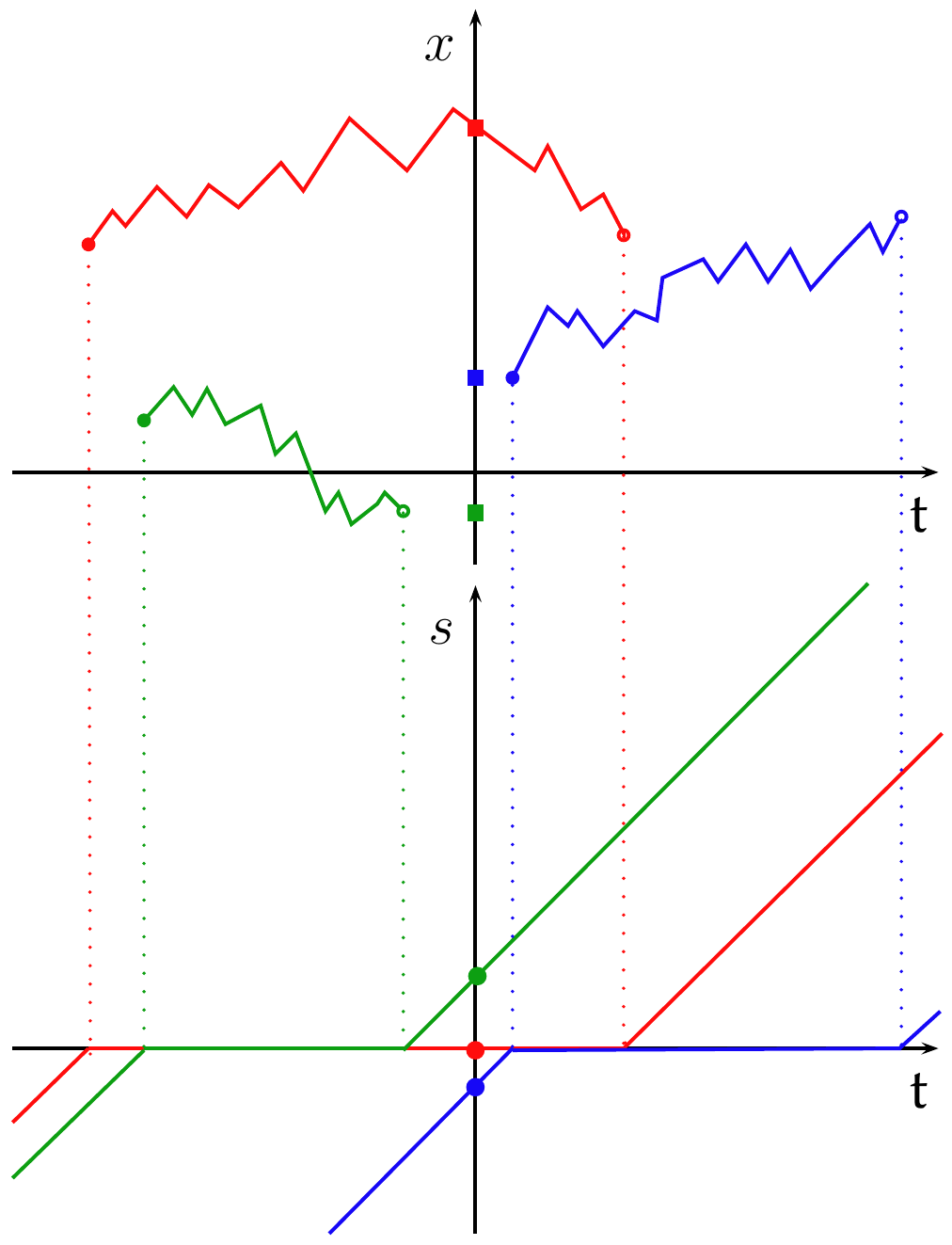}
   \caption{Visualization of three different realizations of
     the Markov process $Z^+$. The top figure shows the dependence of the spatial coordinate $x$ on time $t$. The bottom figure shows the dependence of $s$ on $t$. }\label{fig:particles}
   \end{centering}
\end{figure}

Consider a Markov process $Z^+$ on $\R^2$ which can be described as follows. Suppose that at time $0$ the process $Z^+$ starts at $(x_0,s_0)$ with $s_0<0$ (the particle is not yet born). Then, at any time $t\in (0, |s_0|)$ the particle is still not born meaning that $Z^+(t) = (x_0, s_0+t)$. At time $t=|s_0|$ the particle is born, and it appears on the real line at position $x_0$. After the birth, its coordinate $x$ changes according to the L\'evy process $L^+$, while the time coordinate $s$ remains equal $0$ (meaning that the particle is alive). After an exponential time $\tau^+ \sim \text{Exp}(\theta_-)$, the particle is killed while being located at some spatial position denoted by $x_1 = x_0+L^+(\tau^+)$. After the killing, the particle disappears. Formally, this means that its spatial coordinate $x=x_1$ remains constant, whereas the time coordinate $s$ increases at unit rate as the time goes on.
To summarize, if the process $Z^+$ starts at time $0$ at $(x_0,s_0)\in\R^2$ with $s_0<0$, then
$$
Z^+(t)
=
\begin{cases}
(x_0,s_0+t), &\text{if } s_0+t\leq 0, \\
(x_0+L^+(s_0+t),0), &\text{if } 0\leq s_0+t\leq \tau^+,\\
(x_0+L^+(\tau^+),s_0+t-\tau^+), &\text{if } \tau^+ \leq s_0+t.
\end{cases}
$$
The description of $Z^+$ in the cases $s_0=0$ and $s_0>0$ is similar. Let $P_t^+$, $t\geq 0$, be the probability transition kernel of the process $Z^+$.
\begin{theo}\label{theo:invar_measure}
The following $\sigma$-finite measure on $\R^2$  is invariant for the Markov process $Z^+$:
\begin{align}\label{intensity}
\nu(\dd x, \dd s)
=
\eee^{-x}\dd x \,
\left(
\ind_{s < 0} \theta_+ \dd s + \delta_0(s) + \ind_{s > 0} \theta_- \dd s \right),
\end{align}
where $\theta_+$ and $\theta_-$ satisfy \eqref{rates}.
\end{theo}
\begin{proof}
In the sequel, we write $\mu(\dd x) = \eee^{-x}\dd x$, $x\in\R$.
Fix some time $t>0$. Let $B\subset \R^2$ be a Borel set. We need to verify that
$$
\int_E P^+_t( (x,s), B ) \nu(\dd x, \dd s) = \nu(B).
$$
\vspace*{2mm}
\noindent
\textsc{Case 1.} In the case $B\subset \R\times (-\infty, 0)$, we obtain
\begin{align*}
\int_E P^+_t( (x,s), B ) \nu(\dd x, \dd s)
=
\int_{B-(0,t)} \nu(\dd x, \dd s)
=
\nu(B).
\end{align*}

\vspace*{2mm}
\noindent
\textsc{Case 2.}
In the case when $B = B_0\times\{0\}$, where $B_0\subset \R$ is a Borel set, we obtain
  \begin{equation}\label{u_zero_1}
    \int_E P^+_t((x,s), B) \nu(\dd x, \dd s)
    =
    \int_{-t}^{0} \int_\R q_{t+s}^+(x,B_0)\theta_+ \mu(\dd x)\dd s + \int_\R  q_{t}^+(x,B_0) \mu(\dd x).
  \end{equation}
The second summand on the right-hand side of~\eqref{u_zero_1} equals
\begin{align}\label{non_stat}
\int_\R q_{t}^+(x,B_0) \mu(\dd x)
     %= \eee^{-\theta_- t} \int_\R p_{t}^+(x,B_0) \mu(\dd x) =
=
\eee^{-\theta_- t} \mu(B_0)  \E \eee^{L^+(t)}
=
\eee^{t(\psi(1)-\theta_-)}\mu(B_0)
=
\eee^{-\theta_+ t}\mu(B_0),
\end{align}
where we used \eqref{rates} for the last equality. With this observation, the first summand on the right-hand side of~\eqref{u_zero_1} equals
  \begin{align*}
    \mu(B_0) \int_{-t}^{0} \eee^{-\theta_+(t+s)} \theta_+ \dd s
    %&= \mu(B_0) \frac{\theta_-}{\psi(1) - \theta_+}\left( \eee^{(\psi(1) - \theta_+)t} - 1 \right)\\
    &= \mu(B_0) \left( 1 - \eee^{-\theta_+t} \right),
  \end{align*}
Thus, \eqref{u_zero_1} equals $\mu(B_0)=\nu(B)$.

\vspace*{2mm}
\noindent
\textsc{Case 3.}
Consider finally the case $B\subset \R\times (0,+\infty)$. Take some point $(y,u)\in B$. Let first $0<u<t$. We have
\begin{align*}
\lefteqn{\int_E P^+_t((x,s), (\dd y,\dd u)) \nu(\dd x, \dd s)}\\
&=
\int_{u-t}^{0} \theta_+ \dd s  \int_\R \mu(\dd x)  q_{t-u+s}^+(x,\dd y) \theta_-\dd u
+
\int_\R  \mu(\dd x) q_{t-u}^+(x,\dd y) \theta_- \dd u \\
&=
\theta_- \left( 1 - \eee^{-\theta_+(t-u)} \right) \mu(\dd y) \dd u
+
\theta_- \eee^{-\theta_+(t-u)} \mu(\dd y) \dd u\\
&=\theta_-  \mu(\dd y) \dd u.
\end{align*}
In the case $u>t$ the computation is the same as in Case~1.
\end{proof}
We are now going to define the dual of the process $Z^+$ w.r.t.\ the invariant measure $\nu$. Replacing in the definition of $Z^+$ the killing rate by $\theta_+$ and the driving process by $L^-$, and reversing the time direction, we obtain another Markov process on $\R^2$ denoted by $Z^-$. For example, if the process $Z^-$ starts at time $0$ at $(x_0,s_0)\in\R^2$ with $s_0 > 0$, then in the first stage the coordinate $s$ decreases linearly at unit rate to $0$, in the second stage the coordinate $s$ stays $0$, while the coordinate $x$ changes according to a L\'evy process $L^-$ until its killing after time $\tau^-\sim \text{Exp}(\theta_+)$, and finally in the third stage the coordinate $s$ decreases linearly at unit rate while the coordinate $x$ stays constant. More precisely, we have
$$
Z^-(t)
=
\begin{cases}
(x_0,s_0-t), &\text{if }  t - s_0\leq 0, \\
(x_0 + L^-(t-s_0),0), &\text{if } 0\leq t-s_0\leq \tau^-,\\
(x_0+L^-(\tau^-),\tau^- - t + s_0), &\text{if } \tau^- \leq t - s_0.
\end{cases}
$$
The probability transition kernel of $Z^-$ is denoted by $P^-$.
\begin{theo}\label{theo:duality}
The Markov processes $Z^+$  and $Z^-$ are in duality (in the sense of~\eqref{duality}) w.r.t.\ the invariant measure $\nu$.
\end{theo}
\begin{proof}
We need to establish the equality
\begin{align}\label{mu_duality}
P_t^-((x,s), \dd(y,u)) \, \nu(\dd x, ds) = P_t^+((y,u), \dd(x,y))\, \nu(\dd y, \dd u),
\end{align}
for any $(x,s),(y,u)\in \R^2$.
With regard to the definitions of $Z^+$ and $Z^-$,
showing \eqref{mu_duality} breaks down to several cases depending
on the signs of $s$ and $u$. We exemplarily consider the case $s > 0$ and $u<0$. Let $t \geq s-u$ because otherwise the transition density is $0$. We have
\begin{align*}
    P_t^-((x,s), \dd(y,u))\, \nu(\dd x, \dd s)
    &= \theta_+ q_{t-s+u}^-(x,\dd y) \dd u \, \eee^{-x} \dd x \,\theta_- \dd s\\
    &= \theta_- q_{t-s+u}^+(y,\dd x) \dd u \, \eee^{-y} \dd y \, \theta_+ \dd s\\
    &= P_t^+((y,u), \dd(x,s))\, \nu(\dd y, \dd u).
  \end{align*}
The second equality uses the duality relation~\eqref{eq:q_+_q_-_dual}.
\end{proof}

\begin{proof}[Proof of Theorem~\ref{theo:inv_two_sided_PPP}]
Consider particles on $\R^2$ forming a Poisson point process on $\R^2$ with intensity $\nu$ defined in Theorem~\ref{theo:invar_measure}. Let the forward motion of the particles be given by the independent Markov processes $Z^+$, whereas the backward motion of particles be given by the independent Markov processes $Z^-$. By Theorem~\ref{theo:duality}, the  Markov processes $Z^+$ and $Z^-$ are in duality w.r.t.\ the measure $\nu$. By Theorem~\ref{thm1}, the resulting system of processes is stationary on $\R^2$. Discarding the coordinate $s$ (responsible for the ``age'' of the particles) and putting the
spatial coordinate $x$ to $-\infty$ whenever $s \neq 0$, we obtain the same particle system as described in Theorem~\ref{theo:inv_two_sided_PPP}.

By the stationarity of the particle system, the left-hand side of~\eqref{eq:intensity_kuznetsov} is shift-invariant, whereas the right-hand side is shift-invariant by definition. Hence, when proving~\eqref{eq:intensity_kuznetsov}, there is no restriction of generality in assuming that $t_1=0$. But in this case, the functions $\tilde V_i^+$ and $\tilde V_i^-$ make no contribution to the intensity of $\Pi$ on the left-hand side of~\eqref{eq:intensity_kuznetsov}. The contribution of the functions $V_i$ is given, by the transformation formula for the PPP, by the right-hand side of~\eqref{eq:intensity_kuznetsov}.

The stationarity of the process $\eta$ in~\eqref{eq:eta_with_killing_def} now follows immediately from the above. The max-stability condition~\eqref{eq:max_stable_def} can be verified by noting that the union of $n$ independent copies of the PPP $\pi_0$ (similarly, $\pi_+,\pi_-$) has the same intensity as the original process spatially shifted by $\log n$.
\end{proof}

\section{Proofs: General properties and extremal index}\label{sec:proof_general_properties}

\subsection{Proof of Proposition~\ref{prop_finite}}\label{subsec:proof_prop_finite}
We follow the idea used in the proof of Proposition 13 in~\cite{kab2009}. By stationarity, we can assume that $K=[0,T]$. Then, the paths $\tilde V_i^{-}$ make no contribution to the process $\eta$ on $K$. Fix some $l\in\N$.
For $k\in\N$ consider the random event
$$
A_{l,k} = \left\{\inf_{t\in [0,T]} \max_{i=1,\ldots,l} (U_i + L_i(t))> -k\right\}.
$$
Clearly, the event $A_l:=\cup_{k\in\N} A_{l,k}$ has probability at least $1-(1-\eee^{-T \theta_-})^l$ (which is the probability that at least one of the paths $U_i+L_i(t)$, $1\leq i\leq l$, will not be killed in $[0,T]$). On the event $A_{l,k}$, the set $J$ is contained in $I_k\cup \tilde I_k^+$, where
\begin{align*}
I_k &= \left\{i\in\N\colon U_i+ \sup_{t\in [0,T]} L_i(t) > -k\right\},
\\
\tilde I_k^+ &= \left\{i\in\N\colon \tilde U_i^+ +\sup_{t\in [0,T]} \tilde L_i^+(t) > -k, \tilde T_i^+ \in [0,T]\right\}.
\end{align*}
The cardinality of $I_k$ has Poisson distribution with parameter
\begin{equation}\label{eq:lambda_k}
\lambda_k := \int_{\R} \eee^{-u}\P\left[\sup_{t\in [0,T]} L(t) > u-k\right] \dd u
=
\eee^{k} \int_{\R} \eee^{v} \P\left[\sup_{t\in [0,T]} L(t) > v \right] \dd v.
\end{equation}
By an inequality of Willekens (see Equation 2.1 in~\cite{willekens}), we have the estimate
$$
\P\left[\sup_{t\in [0,T]} L(t) > v \right] \leq C \P[L(T)>v-1]
$$
for all $v>1$ and some finite constant $C$. Since $\E \eee^{L(T)} = 1$, the integral on the right-hand side of~\eqref{eq:lambda_k} converges and $\lambda_k$ is finite. It follows that the set $I_k$ is finite a.s.\ on the event $A_{l,k}$. Similarly, the set $\tilde I_k$ is finite a.s.\ on $A_{l,k}$. It follows that the set $J$ is finite a.s.\ on $A_l$. But we can make the probability of $A_l$ as close to $1$ as we wish by choosing appropriately large $l$.

\subsection{Proof of Proposition~\ref{prop:extremal_corr}}\label{subsec:proof_extremal_corr}
The event $\{\eta(0)<0, \eta(t)<0\}$ occurs if and only if the following $3$ independent events occur simultaneously:

\begin{enumerate}
\item there is no $i\in\N$ such that $U_i>0$;
\item there is no $i\in\N$ such that $U_i<0$ but $U_i + L_i^+(t)>0$;
\item there is no $i\in\N$ such that $s:=\tilde T_i^+ \in [0,t]$ and $\tilde U_i^+ +L^+_i(t-s)>0$.
\end{enumerate}
With $v:=t-s$ we have
\begin{align*}
\lefteqn{-\log \P[\eta(0)<0, \eta(t)<0]}\\
&=
1
+
\eee^{-\theta_- t} \int_{-\infty}^{0} \eee^{-u} \P[\xi(t)>-u] \dd u
+
\theta_+\int_0^t \eee^{-\theta_- v } \int_{-\infty}^{+\infty} \eee^{-u} \P[\xi(v) > -u] \dd u \dd v\\
&=
1
+
\eee^{-\theta_- t} \int_{0}^{\infty} \eee^{u} \P[\xi(t)>u] \dd u
+
\theta_+\int_0^t \eee^{-\theta_- v } \E \eee^{\xi(v)} \dd v
.
\end{align*}
Since $\E \eee^{\xi(v)} = \eee^{\psi(1) v}$ by~\eqref{psi} and $\psi(1) = \theta_--\theta_+$ by~\eqref{rates}, the third term is equal to $1-\eee^{-\theta_+ t}$. The desired formula~\eqref{eq:extremal_corr} follows easily.

\subsection{Proof of Theorem~\ref{theo:extremal_index}}\label{subsec:proof_extremal_index}
First, we compute $\Theta(T)$ as defined by~\eqref{eq:Theta_T_def}. Recall that $L^+$ is the process obtained by killing $\xi$ with rate $\theta_-$. Write $M(t):=\sup_{u\in [0,t]} L^+(u)$. Then, by the definition of L\'evy--Brown--Resnick processes given in Section~\ref{subsec:levy_brown_resnick_general}, we have
\begin{align*}
\lefteqn{-\log \P\left[\sup_{t\in[0,T]} \eta(t) \leq x\right]}\\
&=
\int_{-\infty}^{+\infty}\eee^{-u} \P[M(T)>x-u] \dd u + \theta_+ \int_0^T \int_{-\infty}^{+\infty} \eee^{-u} \P[M(s)>x-u]\dd u \dd s\\
&=
\eee^{-x} \left(\E \eee^{M(T)} + \theta_+\int_0^T \E \eee^{M(s)}\dd s\right).
\end{align*}
Note that the first integral is the contribution of particles which are present at time $0$, whereas the second integral is the contribution of particles born at $T-s$, where $s\in [0,T]$.
Writing $f(t)=\E \eee^{M(t)}$, we obtain that
\begin{equation}\label{eq:Theta_T_formula}
\Theta(T)=f(T)+\theta_+ \int_0^T f(s) \dd s.
\end{equation}
We determine the behavior of $\Theta(T)$ as $T\to\infty$.

\vspace*{2mm}
\noindent
\textsc{Case 1.} Let $\theta_+=0$. We will prove that
\begin{equation}\label{eq:asympt_f(T)_case1}
\Theta(T) = f(T)\sim \psi'(1) T \text{ as } T\to +\infty.
\end{equation}
Let $\tau(\lambda)\sim \text{Exp}(\lambda)$ be random variable which has an exponential distribution with parameter $\lambda>0$ and is independent of everything else. Then,
$$
\int_0^\infty f(T) \lambda \eee^{-\lambda T} \dd T = \E f(\tau(\lambda)) = \E \eee^{\sup_{s\in [0,\tau(\lambda)]} L^+(s)} = \E \eee^{\sup_{s\in [0,\tau(\lambda+\theta_-)]} \xi(s)}.
$$
Since $\xi$ is a L\'evy process with no positive jumps, Corollary~2 on page~190 of~\cite{ber1996} states that
$$
\sup_{s\in [0,\tau(\lambda+\theta_-)]} \xi(s) \sim \text{Exp}(\psi^{-1}(\lambda + \theta_-)).
$$
It follows that
\begin{equation}\label{eq:laplace_transf_f}
\int_0^\infty f(T) \eee^{-\lambda T} \dd T = \frac{\psi^{-1}(\lambda + \theta_-)}{\lambda (\psi^{-1}(\lambda + \theta_-)-1)}
\sim
\frac{\psi'(1)}{\lambda^2}
\text{ as } \lambda\downarrow 0,
\end{equation}
where in the last step we used that $\psi(1)=\theta_-$, see~\eqref{rates}, and hence, $\psi^{-1}(\theta_-)=1$. Since the function $f$ is non-decreasing, we can apply to~\eqref{eq:laplace_transf_f} the standard Tauberian theory, see Theorem 4 on page 423 in~\cite{feller_book},  to conclude that~\eqref{eq:asympt_f(T)_case1} holds. This proves the first case of~\eqref{eq:Theta}.

\vspace*{2mm}
\noindent
\textsc{Case 2.} Let $\theta_+>0$. We will prove that
\begin{equation}\label{eq:asympt_f(T)_case2}
C:=\lim_{T\to\infty} f(T) = \E \eee^{\sup_{s>0} L^+(s)} = \frac{\psi^{-1}(\theta_-)}{\psi^{-1}(\theta_-)-1}<\infty.
\end{equation}
The equality of the limit and the expectation follows from the monotone convergence theorem. We have  to compute the expectation. Since $L^+$ is obtained from $\xi$ by killing it with rate $\theta_-$ and since  $\xi$ is a L\'evy process with no positive jumps, we can again use Corollary~2 on page~190 of~\cite{ber1996} to obtain that
$$
\sup_{s>0} L^+(s) \sim \text{Exp} (\psi^{-1}(\theta_-)).
$$
Note that $\psi^{-1}(\theta_-)>1$ because $\psi(1)=\theta_--\theta_+ < \theta_-$ by~\eqref{rates}. Recalling the Laplace transform of the exponential distribution, we obtain~\eqref{eq:asympt_f(T)_case2}.
Together with~\eqref{eq:Theta_T_formula} this clearly implies that $\Theta(T)\sim C \theta_+ T$ as $T\to\infty$. The proof of the second case of~\eqref{eq:Theta} is complete.

\section{Proofs: Convergence results}\label{sec:proof_convergence}
\subsection{Proof of Theorem~\ref{theo:BR_for_levy}}\label{subsec:proo_brown_resnick_levy}
\textsc{Step 1.}
For $n\in\N$ define i.i.d.\ random variables $U_{1,n}, \ldots, U_{n,n}$ and i.i.d.\ stochastic processes $L_{1,n}, \ldots, L_{n,n}$  by%i.i.d.\ stochastic processes
\begin{align}
U_{i,n}
&=
\theta (\xi_i(s_n) - b_n),\\
L_{i,n}(t)
&=
\theta (\xi_i(s_n+t)-\xi_i(s_n)) -\psi(\theta) t, \quad t\in\R.
\end{align}
If we restrict the $L_{i,n}$'s to the positive half-axis $t\geq 0$, then the $U_{i,n}$'s are independent of the $L_{i,n}$'s and the  $L_{i,n}$'s are i.i.d.\ copies of the process
$$
L^+(t)=\theta \xi(t) - \psi(\theta) t.
$$
On the other hand, let $\sum_{i=1}^{\infty} \delta_{U_i}$ be a PPP on $\R$ with intensity $\eee^{-u}\dd u$. Independently, let $L_1,L_2,\ldots$ be i.i.d.\ copies of $L$, a two-sided extension of the one-sided L\'evy process $L^+$; see~\eqref{twosided}.
We have to show that weakly on $D(\R)$,
$$
\left\{\eta_n(t):=\max_{i=1,\ldots,n} (U_{i,n} + L_{i,n}(t))\right\}_{t\in\R} \toweak \left\{\eta(t):=\max_{i=1,\ldots,n}(U_i+L_i(t))\right\}_{t\in\R}.
$$
It is known that weak convergence on $D(\R)$ is implied by the weak convergence on $D[-T,T]$ for every $T>0$; see~\cite[Theorem~16.7]{billingsley_book}.
Fix some $T>0$. We proceed as follows. In Steps~2--5 we will prove weak convergence on the space $D[0,T]$. The two-sided convergence on $D[-T,T]$ will be established in Step~6.

\vspace*{2mm}
\noindent
\textsc{Step 2.}
We prove that the point process $\sum_{i=1}^n \delta_{U_{i,n}}$ converges weakly to $\sum_{i=1}^{\infty} \delta_{U_i}$, as $n\to\infty$. The point processes are considered on the state space $(-\infty, +\infty]$.
By~\cite[Proposition 3.21]{res2008}, it suffices to show that  for every $u\in\R$,
\begin{equation}\label{eq:asympt0}
\lim_{n\to\infty} n \P[U_{1,n}>u] = \eee^{-u}.
\end{equation}
By the precise large deviations theorem of Bahadur--Rao--Petrov~\cite{petrov}, we have
\begin{equation}\label{eq:petrov}
\P[\xi(t)>\beta t] \sim \frac{1}{I'(\beta) \sqrt{2\pi \psi''(I'(\beta)) t}} \eee^{-I(\beta) t},
\quad t\to +\infty,
\end{equation}
uniformly in $\beta$ as long as it stays in a compact subinterval of $(\beta_0,\beta_\infty)$.
Let $\beta_n= (b_n+\frac u\theta)/s_n$, so that $\lim_{n\to\infty} \beta_n = \psi'(\theta)\in (\beta_0,\beta_\infty)$ by~\eqref{eq:b_n_def}. Note that $\lim_{n\to\infty} I'(\beta_n)=\theta$ because $I'$ is the inverse function of $\psi'$. It follows from~\eqref{eq:petrov} that
\begin{equation}\label{eq:asympt1}
\P[U_{1,n}>u]
=
\P[\xi(s_n) > \beta_n s_n]
\sim
\frac{1}{\theta \sqrt{2\pi \psi''(\theta) s_n}} \eee^{-I(\beta_n)s_n}
,\quad n\to\infty.
\end{equation}
Using the definitions of $\beta_n$ and $b_n$ (see~\eqref{eq:b_n_def}) together with Taylor's expansion, we obtain that
\begin{align*}
I(\beta_n) s_n
=
I\left(I^{-1} (\lambda_n) + \frac{u}{\theta s_n}\right) s_n
=
\log n - \log(\theta \sqrt{2\pi \psi''(\theta)s_n}) + u + o(1),
\end{align*}
where we used the fact that $I^{-1}(\lambda_n)$ converges to $I^{-1}(\lambda) = \psi'(\theta)$, see~\eqref{eq:b_n_def}, and that $I'(\psi'(\theta))=\theta$. Inserting this into~\eqref{eq:asympt1}, we obtain the required Equation~\eqref{eq:asympt0}. At this point note the following consequence of~\eqref{eq:asympt0}:
\begin{equation}\label{eq:U_i_n_to_Gumbel}
\max_{i=1,\ldots,n} U_{i,n} \todistr \eee^{-\eee^{-u}}.
\end{equation}

\vspace*{2mm}
\noindent
\textsc{Step 3.}
For a truncation parameter $a \in \N$ we define the truncated versions of the processes $\eta_n$ and $\eta$ by
\begin{align}\label{eq:M_n_trunc_eta_trunc}
\eta_{n}^{(a)}(t) = \max_{\substack{i = 1,\dots, n\\ U_{i,n} > -a}} (U_{i,n}+L_{i,n}(t)),
\quad
\eta^{(a)}(t) = \max_{\substack{i = 1,\dots, n\\ U_{i} > -a}} (U_i + L_i(t)).
\end{align}
We prove that for every fixed $a\in\N$, the process $\eta_n^{(a)}$ converges to $\eta^{(a)}$ weakly on $D[0,T]$. Consider a bounded, continuous function $f:D[0,T]\to \R$. We need to show that
\begin{equation}\label{eq:E_f_M_n_tau_to_E__f_eta_tau}
\lim_{n\to\infty} \E f(\eta_n^{(a)}) = \E f(\eta^{(a)}).
\end{equation}
Let $\mathcal M$ be the space of locally finite integer-valued measures on $\bar \R=(-\infty, +\infty]$. As usually, $\mathcal M$ is endowed with vague topology. Let $\mathcal M_{a}$ be the (open) set of all $\mu\in \mathcal M$ such that $\mu(\{+\infty\}) = \mu(\{-a\})=0$. Define a function $\bar f:\mathcal M_a\to \R$ by
$$
\bar f\left(\sum_{i} \delta_{u_i}\right) = \E f\left(\max_{i\colon u_i>-a} (u_i +L_i(\cdot))\right).
$$
Note that any measure $\mu\in\mathcal M_{a}$ has only finitely many atoms above $-a$. Using this, it is easy to check that the function $\bar f$ is well-defined and continuous on $\mathcal M_{a}$. On $\mathcal M\backslash \mathcal M_a$ we define $\bar f$ to be, say, $0$.  Observe that $\mathcal M_{a}$ has full probability w.r.t.\ the law of the PPP $\sum_{i=1}^{\infty} \delta_{U_i}$. By the continuous mapping theorem, see~\cite[Theorem~2.7]{billingsley_book}, it follows that
$$
\bar f\left(\sum_{i=1}^n \delta_{U_{i,n}}\right)
\todistr
\bar f\left(\sum_{i=1}^{\infty} \delta_{U_i}\right).
$$
It follows that we have the convergence of expectations of these uniformly bounded random variables:
$$
\E f(\eta_n^{(a)})
=
\E \bar f\left(\sum_{i=1}^n \delta_{U_{i,n}}\right)
\ton
\E \bar f\left(\sum_{i=1}^{\infty} \delta_{U_i}\right)
=
\E f(\eta^{(a)}).
$$
This completes the proof of~\eqref{eq:E_f_M_n_tau_to_E__f_eta_tau}.

\vspace*{2mm}
\noindent
\textsc{Step 4.} We prove that $\eta^{(a)}$ converges to $\eta$ weakly on $D[0,T]$, as $a\to +\infty$. It suffices to show that
$$
\lim_{a\to +\infty}\P[\exists t\in [0,T]\colon \eta(t) \neq \eta^{(a)}(t)] =0.
$$
But this follows directly from Proposition~\ref{prop_finite}.

\vspace*{2mm}
\noindent
\textsc{Step 5.} We prove that
$$
\lim_{a\to +\infty}\limsup_{n\to\infty} \P[\exists t\in [0,T]\colon \eta_n(t) \neq \eta_n^{(a)}(t)] = 0.
$$
Define a process $\{\delta_n^{(a)}\colon t\in [0,T]\}$ by
$$
\delta_{n}^{(a)}(t) = \max_{\substack{i = 1,\dots, n\\ U_{i,n} \leq  -a}} (U_{i,n}+L_{i,n}(t)).
$$
It suffices to prove that
\begin{align}
&\lim_{b\to +\infty}\limsup_{n\to\infty} \P\left[\inf_{t\in [0,T]} \eta_n(t) \leq -b\right]=0,\label{eq:aux_lim1}\\
&\lim_{a\to +\infty}\limsup_{n\to\infty} \P\left[\sup_{t\in [0,T]} \delta_n^{(a)}(t) \geq -b\right]=0 \text{ for all } b\in\N. \label{eq:aux_lim2}
\end{align}
\noindent
\textit{Proof of~\eqref{eq:aux_lim1}.}
Let $i_n\in \{1,\ldots,n\}$ be (for concreteness, the smallest) number such that $U_{i_n,n} = \max_{i=1,\ldots,n} U_{i,n}$. Then,
\begin{align*}
\P\left[\inf_{t\in [0,T]}  \eta_n(t) \leq - b\right]
&\leq
\P\left[\inf_{t\in [0,T]} (U_{i_n,n} + L_{i_n,n}(t))\leq -b\right]\\
&\leq
\P\left[U_{i_n,n} \leq -\frac b2\right] + \P\left[\inf_{t\in [0,T]} L(t) \leq -\frac b2\right].
\end{align*}
By~\eqref{eq:U_i_n_to_Gumbel}, the first term on the right-hand side converges, as $n\to\infty$, to $\exp(- \eee^{b/2})$,
which, in turn, converges to $0$ as $b\to +\infty$. The second term on the right-hand side does not depend on $n$ and converges to $0$ as $b\to +\infty$.

\vspace*{2mm}
\noindent
\textit{Proof of~\eqref{eq:aux_lim2}.} Let $S_n = \sup_{t\in [0,T]} L_{1,n}(t)$ and denote by $\mu$ the probability distribution of $S_n$ (which does not depend on $n$).  Note that $S_n$ and $U_{1,n}$ are independent. We have
$$
\P\left[\sup_{t\in [0,T]} \delta_n^{(a)}(t) \geq - b\right]
\leq
n \P\left[U_{1,n}\leq - a,  U_{1,n} + S_n \geq - b \right].
$$
In order to prove~\eqref{eq:aux_lim2} it suffices to show that for some $\eps>0$ and all $b\in\N$,
\begin{align}
&\lim_{n\to\infty}  n\P\left[S_n >  (1+\eps)\log n - b\right] = 0,\label{eq:aux_lim3}\\
& \lim_{a\to+\infty} \limsup_{n\to\infty} n \int_{a-b}^{(1+\eps)\log n - b} \P\left[U_{1,n}> -b - s\right] \mu(\dd s)=0. \label{eq:aux_lim4}
\end{align}
\noindent \textit{Proof of~\eqref{eq:aux_lim3}.} By a result of~\citet{willekens}, the following estimate is valid for all $u>1$:
$$
\P[S_n > u] = \P\left[\sup_{t\in [0,T]} L^+(t) > u\right] \leq C \P[L^+(T) > u-1].
$$
Using this estimate, the fact that $\E \eee^{L^+(T)}=1$ and the Markov inequality, we immediately obtain that~\eqref{eq:aux_lim3} holds for all $\eps>0$ and $b\in\N$.

\vspace*{2mm}
\noindent \textit{Proof of~\eqref{eq:aux_lim4}.} We have
$$
\P\left[U_{1,n} > -b - s\right]
=
\P\left[\xi(s_n) > b_n- \frac 1\theta\left(b + s\right)\right]
=
\P\left[\xi(s_n) > s_n \beta_n (s) \right]
$$
with
$$
\beta_n(s) = \frac{b_n}{s_n}  - \frac{b + s}{\theta s_n}.
$$
Suppose that $s$ stays in the range between $a-b$ and $(1+\eps)\log n - b$. By~\eqref{eq:b_n_def} and~\eqref{growth}, for every $\delta>0$ we have, for sufficiently large $n$,
$$
\beta_n(s)
\geq
\psi'(\theta) - \frac{(1+\eps)\lambda}{\theta} -\delta
=
\frac{\psi(\theta)}{\theta} - \frac{\eps \lambda}{\theta} -\delta,
\quad
\beta_n(s) \leq \psi'(\theta)+\delta.
$$
Since $\beta_0 < \frac{\psi(\theta)}{\theta} < \psi'(\theta) < \beta_\infty$ by convexity of $\psi$, we can take $\eps,\delta>0$ so small and $n$ so large that $\beta_n(s)$ stays in a compact subinterval of $(\beta_0,\beta_\infty)$. Then, we can use the uniformity in~\eqref{eq:petrov}. By convexity of $I$, we have
\begin{align*}
I(\beta_n(s)) s_n
&= I\left(I^{-1} (\lambda_n) -\frac{b + s}{\theta s_n}\right) s_n
\geq
\lambda_n s_n  - I'(I^{-1}(\lambda_n)) \frac{b + s}{\theta}\\
&\geq
\log n - \log(\theta \sqrt{2\pi \psi''(\theta)s_n}) - (b + s),
\end{align*}
where we used that $I'(I^{-1}(\lambda_n)) < I'(I^{-1}(\lambda)) = \theta$. Using the uniformity in~\eqref{eq:petrov} we obtain the estimate
$$
n\P\left[U_{1,n} > -b -s\right]
\leq
\frac{Cn}{\sqrt {s_n}} \eee^{-I(\beta_n(s)) s_n}
\leq
C \eee^{b + s}.
$$
It follows that
$$
\limsup_{n\to\infty} n \int_{a-b}^{(1+\eps)\log n - b} \P\left[U_{1,n}> -b - s\right] \mu(\dd s)
\leq
\int_{a-b}^{\infty}  C \eee^{b + s} \mu(\dd s) = C\eee^{b} \E [\eee^{S_n}\ind_{S_n>a-b}].
$$
Since the law of $S_n$ does not depend on $n$ and  $\E \eee^{S_n}<\infty$, we obtain that the right-hand side goes to $0$ as $a\to +\infty$. This completes the proof of~\eqref{eq:aux_lim4}.

\vspace*{2mm}
Taken together, the results of Steps 3, 4, 5 imply that $\eta_n$ converges to $\eta$  weakly on $D[0,T]$; see~\cite[Theorem~3.2 on p.~28]{billingsley_book}.

\vspace*{2mm}
\noindent
\textsc{Step 6.} Finally, we prove weak convergence on the two-sided space $D[-T,T]$. Consider a modified sequence $\tilde s_n = s_n-T$ which also satisfies~\eqref{growth}. The corresponding sequence $\tilde b_n$ is given by
\begin{equation}\label{eq:tilde_b_n}
\tilde b_n = I^{-1} \left( \frac{\log n - \log(\theta\sqrt{2\pi \psi''(\theta)\tilde s_n})}{\tilde s_n} \right) \tilde s_n = b_n - \frac{\psi(\theta)}{\theta} T +o(1),
\end{equation}
where we used the Taylor expansion. By Steps 1--5 we have, weakly on $D[0,2T]$,
$$
\left\{\max_{i=1,\ldots,n} \xi_i (\tilde s_n + \tilde t) - \tilde b_n\colon \tilde t\in [0,2T]\right\}
\toweak
\left\{\frac 1 \theta \eta(\tilde t) + \frac{\psi(\theta)}{\theta} \tilde t\colon \tilde t\in [0,2T]\right\}.
$$
Introducing the variable $t=\tilde  t-T$, we can rewrite this as follows: Weakly on $D[-T,T]$,
$$
\left\{\eta_n(t) +b_n - \tilde b_n\colon t\in [-T,T]\right\}
\toweak
\left\{\frac 1 \theta \eta(t+T) + \frac{\psi(\theta)}{\theta} (t+T)\colon t\in [-T,T]\right\}.
$$
Using~\eqref{eq:tilde_b_n} and the stationarity of $\eta$, we obtain the required weak convergence on $D[-T,T]$.

\subsection{Proof of Theorem~\ref{theo:BR_for_alpha_stable_OU}}\label{subsec:proof_OU}
Fix $T>0$. Let first $\alpha\neq 1$. Let $\xi_{1,\alpha}, \xi_{2,\alpha},\ldots$ be i.i.d.\ copies of the $\alpha$-stable L\'evy process $\xi_\alpha$.
Consider the process
\begin{equation}\label{eq:tilde_eta_n_alpha}
\tilde \eta_{n,\alpha} (t)
:=
 (\log n)^{\frac 1 \alpha} \max_{i=1,\ldots,n} \xi_{i,\alpha}(\eee^{t/\log n}) - b_{n,\alpha}.
\end{equation}
Let $\gamma_n(t)$ be a function such that
\begin{equation}\label{eq:gamma_n_t}
\eee^{\gamma_n(t)/\log n} = 1+ \frac{t}{\log n},
\quad
t\in\R.
\end{equation}
Solving this equation w.r.t.\ $\gamma_n(t)$ and using Taylor's expansion we obtain that
\begin{equation}\label{eq:gamma_n_t_minus_t}
\lim_{n\to\infty} (\gamma_n(t)-t)=0 \text{ uniformly in } t\in [-T,T].
\end{equation}
From~\eqref{eq:brown_resnick_stable_near1} (recall also the notation introduced in~\eqref{eq:tilde_eta_alpha}, \eqref{eq:tilde_eta_n_alpha}, \eqref{eq:gamma_n_t}) we know that weakly on $D[-T,T]$,
\begin{equation}
\{\tilde \eta_{n,\alpha}(\gamma_n(t))\colon t\in [-T,T]\} \toweak \{\tilde \eta_{\alpha}(t)\colon t\in[-T,T]\}.
\end{equation}
Since by~\eqref{eq:gamma_n_t_minus_t} the Skorokhod $J_1$-distance between $\tilde \eta_{n,\alpha}(\gamma_n(\cdot))$ and $\tilde \eta_{n,\alpha}(\cdot)$ goes to $0$ as $n\to\infty$, we also have
\begin{equation}\label{eq:aux2}
\{\tilde \eta_{n,\alpha}(t)\colon t\in[-T,T]\} \toweak \{\tilde \eta_{\alpha}(t)\colon t\in [-T,T]\}
\end{equation}
weakly on $D[-T,T]$. Recalling that $Z_{i,\alpha}(t)=\eee^{-t/\alpha} \xi_{i,\alpha}(\eee^t)$ are i.i.d.\ $\alpha$-stable Ornstein--Uhlenbeck processes, consider
\begin{equation}\label{eq:aux1}
(\log n)^{\frac 1 \alpha} \max_{i=1,\ldots,n} Z_{i,\alpha}\left(\frac t {\log n}\right) - b_{n,\alpha}
=
\eee^{-\frac{t}{\alpha\log n}} \tilde \eta_{n,\alpha}(t) + b_{n,\alpha}\left(\eee^{-\frac{t}{\alpha\log n}}-1\right).
\end{equation}
By~\eqref{eq:aux2}, the first term on the right-hand side of~\eqref{eq:aux1} converges to $\tilde \eta_{\alpha}(t)$ weakly on $D[-T,T]$, whereas the second term is deterministic and converges, uniformly on $[-T,T]$, to $-\psi(\theta_\alpha) t/\theta_\alpha$  by~\eqref{eq:b_n_alpha_def} and~\eqref{eq:theta_alpha}. It follows that the right-hand side of~\eqref{eq:aux1} converges to $\tilde \eta_{\alpha}(t)-\psi(\theta_\alpha) t/\theta_\alpha = \eta_{\alpha}(t)/\theta_\alpha$ weakly on $D[-T,T]$.

The proof in the case $\alpha=1$ is similar, but it is based on~\eqref{eq:brown_resnick_stable_near1_alpha1} and uses $\tilde b_{n,1}$ instead of $b_{n,\alpha}$.

\section*{Acknowledgement}
The authors are grateful to Martin Schlather for numerous discussions on the topic of the paper and to an unknown referee for useful comments.
We are further grateful to Steffen Dereich and Leif D\"oring from whom we learned about the connection with Kuznetsov measures~\cite{kuznetsov} and Mitro's construction~\cite{mitro}. A review of these topics can be found in their recent paper~\cite{dereich_doering}.
Financial support by the Swiss National Science Foundation Projects 200021-140633/1, 200021-140686 (first author) is gratefully acknowledged.

\bibliographystyle{plainnat}
\bibliography{Engelke}

\end{document}